\documentclass[12pt]{article}

\usepackage{amsmath, amssymb, amsfonts, url, dsfont, mathrsfs, amsthm}
\usepackage{tikz,graphicx,subfigure,epsfig}
\usepackage{color}
\usepackage{mathptmx}
\usepackage{sidecap}
\usepackage{hyperref}
\usepackage{fancyhdr}

\newlength{\guillotine}
\makeatletter
\renewcommand{\@oddfoot}{\hfill ---~\thepage~--- \hfill}
\makeatother

\renewenvironment{proof}[1][\unskip]{
\par\noindent{\it Proof #1. \ }} { \mbox{}\hfill
$\blacksquare$ \par }

\newtheorem{thm}{Theorem}[section]

\newtheorem{conjecture}[thm]{Conjecture}
\newtheorem{lemma}[thm]{Lemma}

\newtheorem{prop}[thm]{Proposition}

\newtheorem{definition}[thm]{Definition}
\newtheorem{example}[thm]{Example}
\newtheorem{notation}[thm]{Notation}
\theoremstyle{remark}
\newtheorem{rem}[thm]{Remark}

\makeatletter 
\@addtoreset{equation}{section} 
\makeatother

\newcommand\tr{\mathop{\rm tr}\nolimits }
\newcommand\ddh{\mathop{\rm dist}\nolimits_H }
\newcommand\acosh{\mathop{\rm Arcosh}\nolimits }
\newcommand\bbH{\mathord{\mathbb H}}
\newcommand\bbR{\mathord{\mathbb R}}
\newcommand\bbC{\mathord{\mathbb C}}
\newcommand\bbZ{\mathord{\mathbb Z}}
\newcommand\bbN{\mathord{\mathbb N}}
\newcommand\bbD{\mathord{\mathbb D}}

\begin{document}

\title{Zeros of the  Selberg zeta function for symmetric infinite area
hyperbolic surfaces }
\date{ }

\author{Mark Pollicott        \and
        Polina Vytnova\thanks{The authors were partly supported by EPSRC grant
 EP/M001903/1.}
}




\maketitle

\begin{abstract}
In the present paper we give a simple mathematical foundation for
describing the zeros of the Selberg zeta functions~$Z_X$ for certain very
symmetric infinite area surfaces~$X$. For definiteness, we consider the case of
three funneled surfaces. We show that the zeta function is a complex almost
periodic function which can be approximated by complex trigonometric polynomials
on large domains (in Theorem~\ref{thm:main-plus}). As our main application, we
provide an explanation of the striking empirical results of
Borthwick~\cite{B14} (in Theorem~\ref{thm:bor}) in terms of convergence of the
affinely scaled zero sets to standard curves~$\mathcal C$.

\end{abstract}

\section{Introduction}

The Selberg zeta function~$Z_X$ associated to a compact Riemann surface~$X$ with negative
Euler  characteristic  and without boundary is a well known and  much studied
complex function. It is a function of a single complex variable  defined in
terms of the lengths~$\ell(\gamma)$ of the primitive closed geodesics~$\gamma$ on the surface by analogy with the
Riemann zeta function in number theory. 
\begin{definition}
  \label{def:Szeta}
We can formally define the Selberg zeta function by 
\begin{equation}
  Z_X(s) = \prod_{n=0}^\infty \prod_{\stackrel{\gamma = \mbox{\scriptsize{ prime}}}{
  \mbox{\scriptsize { closed geodesics }} }} \left(1-
  e^{-(s+n)\ell(\gamma)}\right), \qquad s \in
\mathbb C,
\label{eq:z1}
\end{equation}
where the product is taken over all 
primitive closed geodesics~$\gamma$ on $X$.
\end{definition}
It was shown by A.~Selberg in 1956 that  for such surfaces  the zeta function
$Z_X$ has an analytic extension  to the entire complex plane and that the
non-trivial zeros can be described in terms of the spectrum  of the
Laplace--Beltrami operator~\cite{S56} and the Selberg trace formula.  
\begin{thm}[Selberg] Let $X$ be a compact Riemann surface with negative Euler
  characteristic and without boundary. Then the function
 $Z_X$ has a simple zero at $s=1$ and for any zero~$s$ in the critical strip $0 <
 \Re(s) <1$ we have that either $s \in [0,1]$ is real, or $\Re(s) = \frac{1}{2}$.  
\end{thm}
A good reference for this standard result is Chapter 2 of the book of Hejhal \cite{hejhal}.

In the case of infinite area surfaces the situation is somewhat different,
since the original trace formulae approach of Selberg no longer  applies.
However, it still follows from the dynamical method of Ruelle that providing  the
surface $X$ appears as a quotient space of the upper half-plane
$\mathbb H^2$ by a convex cocompact Fuchsian group $\Gamma$
then the analogous zeta function $Z_X$ still has an analytic extension  to
the entire complex plane~\cite{PP01},~\cite{R76}.
Unfortunately, this approach provides little effective information on the location
of the zeros, other than that  the first zero at $s=\delta\in(0,1)$, the Hausdorff
dimension of the limit set (of the Fuchsian group).

\begin{figure}[h!]
\centerline{\begin{tikzpicture}[scale=0.35]
\draw [red, dashed] (0,0) ellipse (30mm and 20mm); 
\draw [red, rotate=135, dashed] (0,-10) ellipse (30mm and 20mm);
\draw [red, rotate=45, dashed] (0,10) ellipse (30mm and 20mm);  
\draw (-2,-3.5) .. controls (-2.5,1.5) and (-9,6) .. (-12.9,5.5);
\draw (2,-3.5) .. controls (2.5,1.5) and (9,6) .. (12.9,5.5);
\draw (-7.8,12) .. controls (-3,6) and (3,6) .. (7.8,12);
\end{tikzpicture} }
\caption[A pattern of zeros]{ An artist's impression of a symmetric three funnelled
surface.  The three dashed closed geodesics are assumed to be of equal length
$2b$.} 
\label{fig:draft1}
\end{figure}
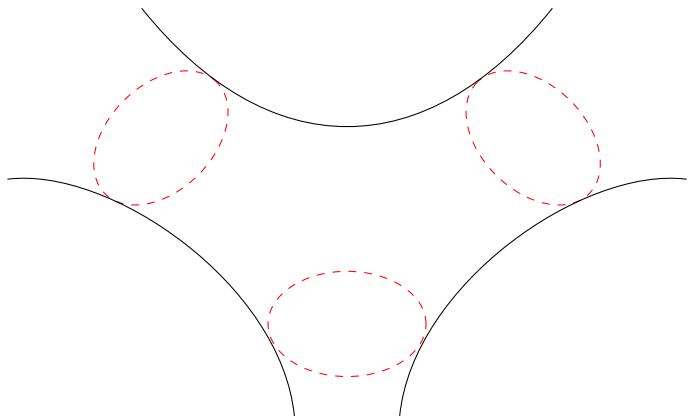

In pioneering   experimental work, D.~Borthwick has studied the location of the
zeros for the zeta function in   specific examples  of infinite area surfaces~\cite{B14}. 
The plot in Figure~\ref{fig:draft2}  is fairly typical for the distribution of
zeros in the critical strip for a symmetric three funnelled surface $X_b$, for large $b$, where 
each of the three simple closed geodesics corresponding to a funnel has the same  length $2b$ 
\footnote{ 
This is an infinite area surface $X_b$ whose compact core $K_b$ corresponds to a sphere
with three disks removed (called a ``pair of pants'') and geodesic boundary
components corresponding to the three unique simple closed  geodesics of length $2b$ around each funnel 
(as represented in Figure~\ref{fig:draft1}). }.

\begin{figure}[h!]
  \centerline{\includegraphics[scale=0.80,angle=0 ]{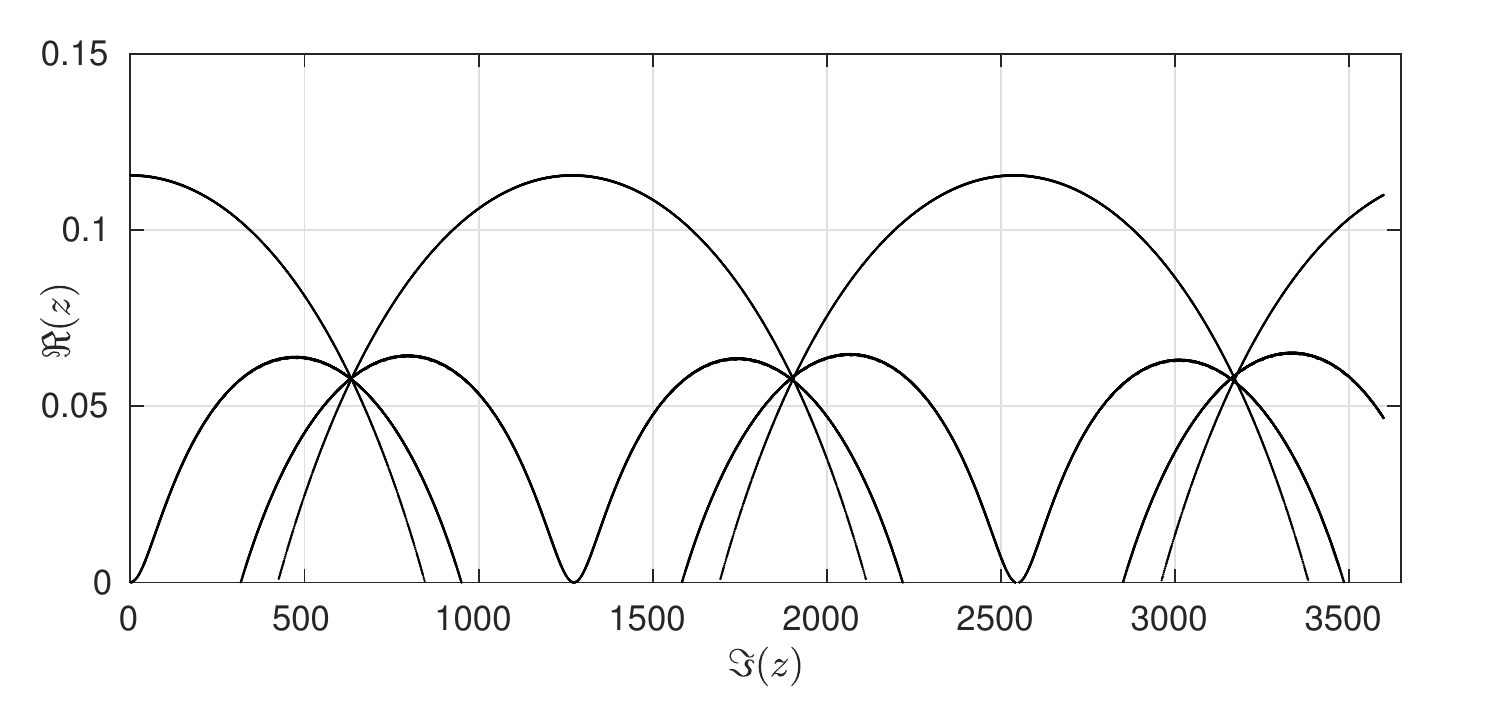} }
\caption[A pattern of zeros]{
 The zeros of the associated zeta function 
$Z_X$ in a critical strip. 
The individual zeros are so close in the plot that it creates the illusion that they lie on 
well defined smooth curves. } 
\label{fig:draft2}
\end{figure}

The  analysis of the Selberg zeta function is via the action of the Fuchsian group on the boundary of hyperbolic space 
and the basic method dates back to work of the first author
over  twenty five years ago~\cite{P91}.
However, it is only with the advent of superior computational
resources, and the ingenuity of those that employ them, that the striking
features seen in Figure~\ref{fig:draft2}, for example,  have been revealed.
A contemporary personal
computer allows one to study symmetric three funneled surface with the length of
the three defining closed geodesics at least $8$ without much difficulty, and this turns out to be a sufficiently
general case. 
We refer the reader to \cite{W15} for more references to the existing literature. 

The main object of study in the present work is the zero set of the function $Z_X$: 
\begin{equation}
  \label{eq:zeroset}
 \mathcal S_X\colon = \{s \in \bbC \mid Z_X(s)=0\}. 
\end{equation}

Carefully studying plots of a few thousands zeros of zeta functions associated
to symmetric 3-funnelled surfaces in a domain $|s| \le 3000$ and
$\Re(s)>0$ for example, one observes that there are certain similarities in a way the zeros are
arranged. We formulate our observations in terms of the theory of almost periodic functions. 
\begin{definition}
  We say that $\tau \in i\mathbb R$ is a $\varepsilon$-translation of a discrete set
  $\{A\} \subset \mathbb C$ if there is a
  bijection $\{A\} \stackrel{\varphi}{\longrightarrow} \{A + \tau\}$ such that
  $d(a,\varphi(a))\le \varepsilon$ for all $a\in \{A\}$. We say that a discrete set is
  almost periodic in the sense of Krein--Levin, if for any
  $\varepsilon>0$ there exists a relatively dense\footnote{The set $\mathcal T \in
  i\mathbb R$ is relatively dense if there exists $l>0$ such that for any
  interval $I\subset i\mathbb R$ of the length $l$ we have that $I \cap \mathcal
  T \ne \varnothing$.} set of $\varepsilon$-translations.
\end{definition}

\noindent{\bf Informal Qualitative Observations.}
Let $X_b$ be the  three funnelled surface defined by three simple closed
geodesics of equal length\footnote{This normalization makes formulae in subsequent
calculations shorter.}~$2b$.
Then for a sufficiently large $b$. 
\begin{description}
\item[O1]:
  The set of zeros $\mathcal S_{X_b}$ {\it appears} to be an almost periodic set in the sense of Krein--Levin,
with relatively dense set of translations $\tau = \{i( \pi k e^b +
\varepsilon_k) \mid k \in \mathbb N \}$, where $\varepsilon_k = o(e^{-b/2})$ as $b \to +\infty$.
\item[O2]:
The set of zeros $\mathcal S_{X_b}$ {\it appears} to lie on a few distinct curves, which {\it seem} to have
a common intersection point at $\frac{\delta}{2} + i \frac{\pi}{2}e^b$, as $b \to +\infty$.
\end{description}
It is well known that the zeta-function can be well approximated by
a sequence of finite exponential sums~\cite{R76},~\cite{JP02}, and therefore the first property is to be
expected: the zero set of a finite exponential sum is almost periodic~(cf.~\cite{L80}, Appendix~1). The
main difficulty here is to identify the set of translations.  

The second property is more mysterious, and this will follow from a specific
approximation of the zeta function by an exponential sum of~$12$ terms on a
large, but bounded, domain, which grows with $b$ exponentially quickly.

In order to provide a rigorous proof, we need to estimate the error term of
approximation of the zeta function by exponential sums as $b \to +\infty$.
This is done in Theorem~\ref{thm:main-plus}. Subsequently using properties of
the zeta function we simplify the exponential sum, which gives the
best approximation, and obtain a function whose zero set belongs to four
sinusoids. 


According to an old result by McMullen~\cite{M98} the largest real zero $\delta$ 
asymptotically behaves like $\frac{\ln 2}{b}$ as $b \to +\infty$, and it
defines the width of the critical strip. Therefore, in the limit $b\to+\infty$
the zero set converes to imaginry axis. However, a suitable affine rescaling allows one to see
the pattern of zeros for large values of $b$. A natural choice for rescaling 
factors is the approximate period of the pattern (in the imaginary direction) and 
approximate reciprocal of the width of the critical strip (in the real direction). 

\begin{notation}
 We will be using the following.
  \begin{enumerate}
\item[(a)]
A compact part of the  critical strip of height $T>0$ which we denote by 
$$
{\mathcal R}(T) = \{ s \in \mathbb C \mid 0 \leq  \Re(s)  \leq
\delta \mbox{ and }   |\Im(s)| \leq T\};
$$
and a compact part of the normalized critical strip of height $S>0$ which we denote by 
$$
\widehat{\mathcal R}(S) = \{ s \in \mathbb C \mid 0 \leq  \Re(s)  \leq
\ln 2 \mbox{ and }   |\Im(s)| \leq S\}.
$$
\item[(b)] We denote the rescaled set of zeros by $$ \widehat{\mathcal S}_{X_b} \colon = 
\left\{\sigma b+i e^{-b} t \Bigl| \sigma+it \in {\mathcal S}_X \right\}.
 $$
 where evidently, $0< \Re(\widehat{\mathcal S}_{X_b}) \le \ln 2$.
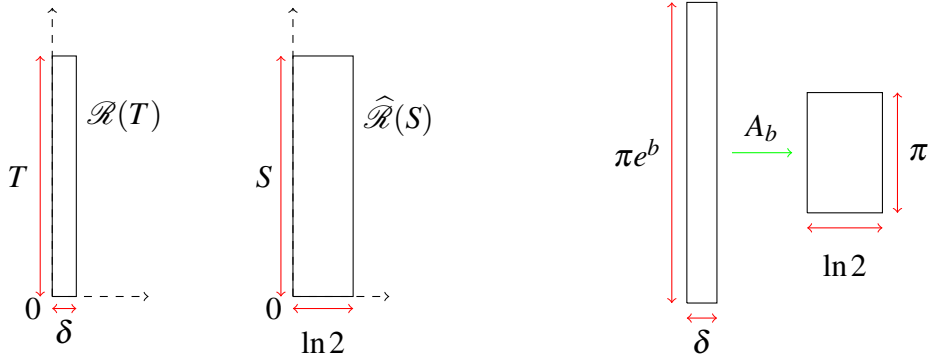
\begin{figure}[h]
\centering
\begin{tikzpicture}[scale=0.32]
\draw (0,-5) -- (1,-5) -- (1,5) -- (0, 5) -- (0,-5);
\draw (10,-5) -- (12.5,-5) -- (12.5,5) -- (10, 5) -- (10,-5);
\draw [red, <->] (0,-5.5) -- (1,-5.5);
\node [below] at (0.5,-5.5) {$\delta$};
\draw [red, <->] (-0.5,-5) -- (-0.5,5);
\node [left] at (-0.5,0.0) {$T$};
\draw [red, <->] (10,-5.5) -- (12.5,-5.5);
\node [below] at (11.25,-6) {$\ln 2$};
\draw [red, <->] (9.5,-5) -- (9.5,5);
\node [right] at (8,0.0) {$S$};
\draw[->, dashed] (0,-5) -- (4,-5); 
\draw[->, dashed] (0,-5) -- (0,7); 
\draw[->, dashed] (10,-5) -- (14,-5); 
\draw[->, dashed] (10,-5) -- (10,7); 
\node [right] at (1,2.5) {$\mathcal R(T)$};
\node [right] at (12.5,2.5) {$\widehat{\mathcal R}(S)$};
\node [left] at (0,-5.5) {$0$};
\node [left] at (10,-5.5) {$0$};
\end{tikzpicture}
\hskip 2cm
\begin{tikzpicture}[scale=0.4]
\draw (0,-5) -- (1,-5) -- (1,5) -- (0, 5) -- (0,-5);
\draw (4,-2) -- (6.5,-2) -- (6.5,2) -- (4, 2) -- (4,-2);
\draw [red, <->] (0,-5.5) -- (1,-5.5);
\node [below] at (0.5,-5.5) {$\delta$};
\draw [red, <->] (-0.5,-5) -- (-0.5,5);
\node [left] at (-0.5,0.0) {$\pi e^b$};
\draw [red, <->] (4,-2.5) -- (6.5,-2.5);
\node [below] at (5.25,-3) {$\ln 2$};
\draw [red, <->] (7,-2) -- (7,2);
\node [right] at (7,0.0) {$\pi$};
\draw [green, ->] (1.5,0) -- (3.5,0);
\node [above] at (2.5,0.0) {$A_b$};
\end{tikzpicture}
\caption{(a)~The strips ${\mathcal R}(T)$ and 
$\widehat {\mathcal R}(S)$; (b) By renormalizing the strip ${\mathcal R}(\pi e^b)$ to 
$\widehat {\mathcal R}(\pi)$ we can compare the zeros of zeta functions for different $b$, as $b$ tends to infinity.
}
\end{figure}

{\rm We now introduce a family of four sinusoidal curves approximating $\widehat{\mathcal S}_{X_b}$ 
as $b \to +\infty$.  }
\item[(c)] Let $\mathcal C=\cup_{j=1}^4 \mathcal C_j$, where
$$
\begin{aligned}
\mathcal C_1 &= \left\{ 
\frac12\ln|2-2\cos(t)| + it \mid t \in \mathbb R \right\}; \cr 
\mathcal C_2 &= 
\left\{ \frac12 \ln|2+2\cos(t)| + it \mid t \in \mathbb R \right\}; \cr
\mathcal C_3 &= \left\{ \frac12 \ln \left|
1 - \frac12 e^{2it} - \frac12 e^{it} \sqrt{4 - 3 e^{2i t}} \right| + it \mid t
\in \mathbb R \right\}; \cr
\mathcal C_4 &= \left\{  \frac12 \ln \left|
1 - \frac12 e^{2it} + \frac12 e^{it} \sqrt{4 - 3 e^{2i t}} \right| + it \mid
t \in \mathbb R \right\}.
\end{aligned}
$$
Note that the curve $\mathcal C$ in Figure~\ref{fig:c-curves} looks similar to empirical plots in
Figure~\ref{fig:draft2}. The apparent intersections in Figure~\ref{fig:draft2}
correspond to the intersections of $\mathcal C_j$:
$$
\bigcap_{j=1}^4 \mathcal C_j = \Bigl\{ \frac{\ln 2}{2}+i\pi \Bigl(\frac12 +
k\Bigr)\Bigr\}, \quad k\in\bbZ.
$$

\end{enumerate}
\end{notation}

\begin{figure}[h]
\centering
\includegraphics[scale=0.3,angle=0 ]{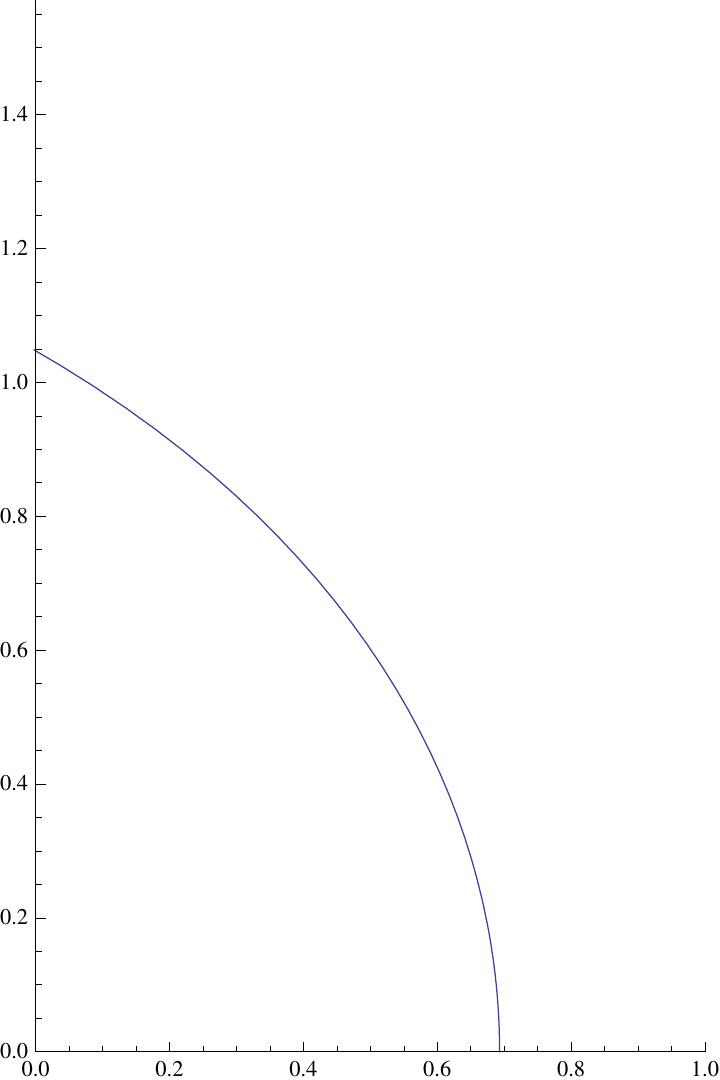} \hskip 0.25cm
\includegraphics[scale=0.3,angle=0 ]{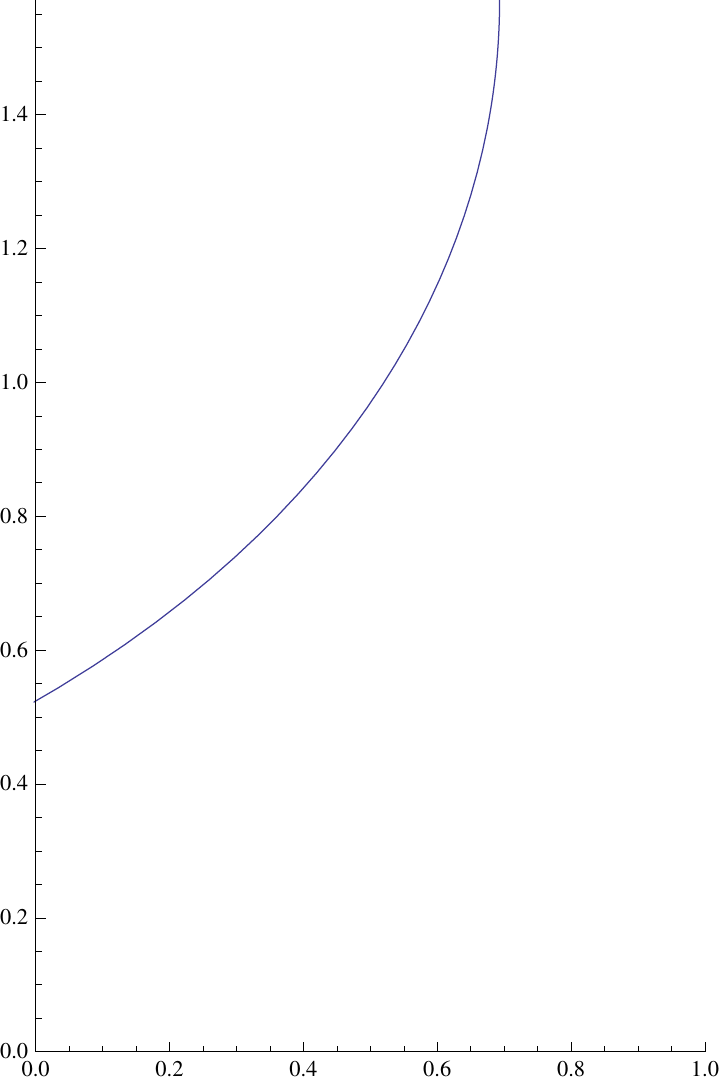}  \hskip 0.25cm
\includegraphics[scale=0.3,angle=0 ]{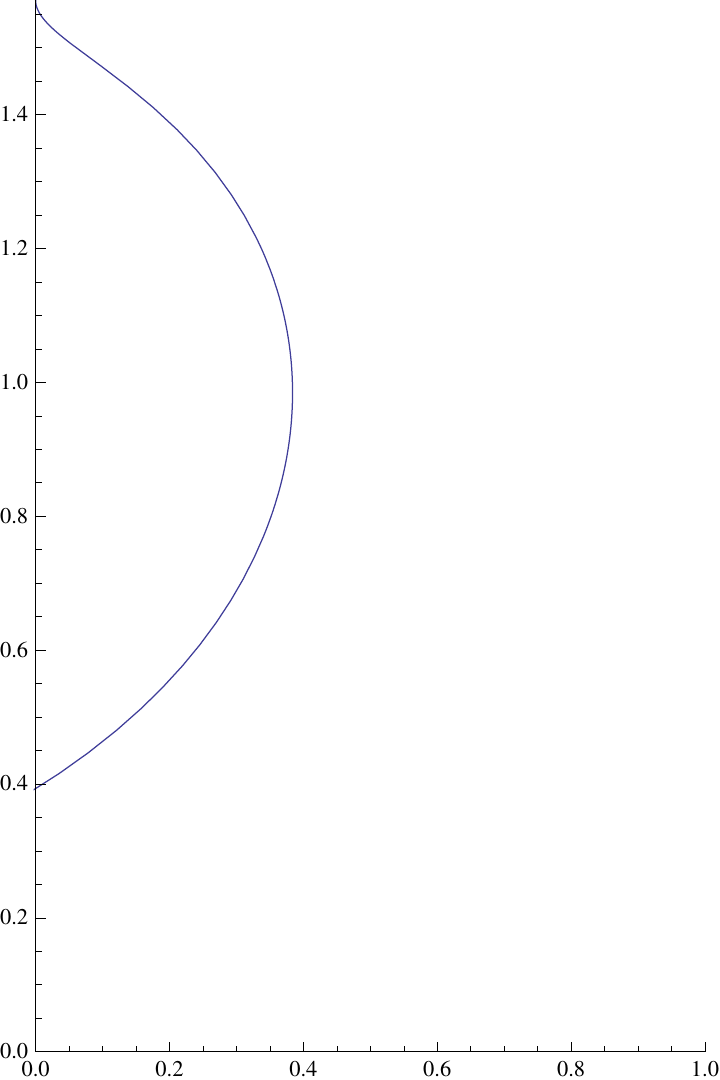}  \hskip 0.25cm
\includegraphics[scale=0.3,angle=0 ]{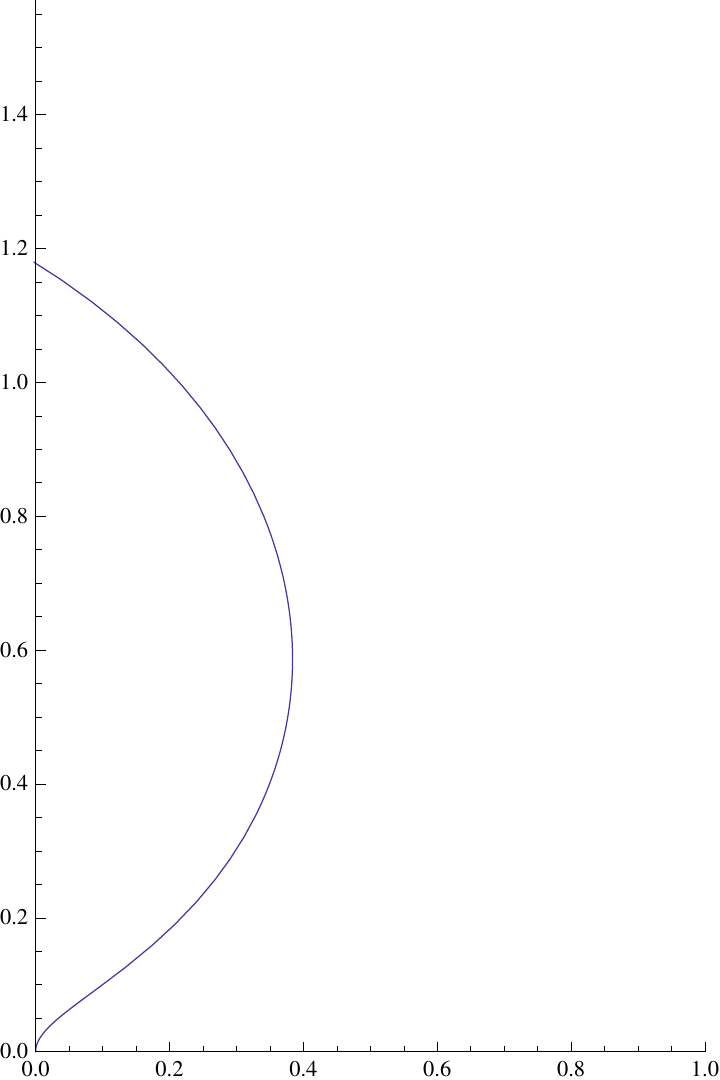}  \hskip 0.25cm
\includegraphics[scale=0.3,angle=0 ]{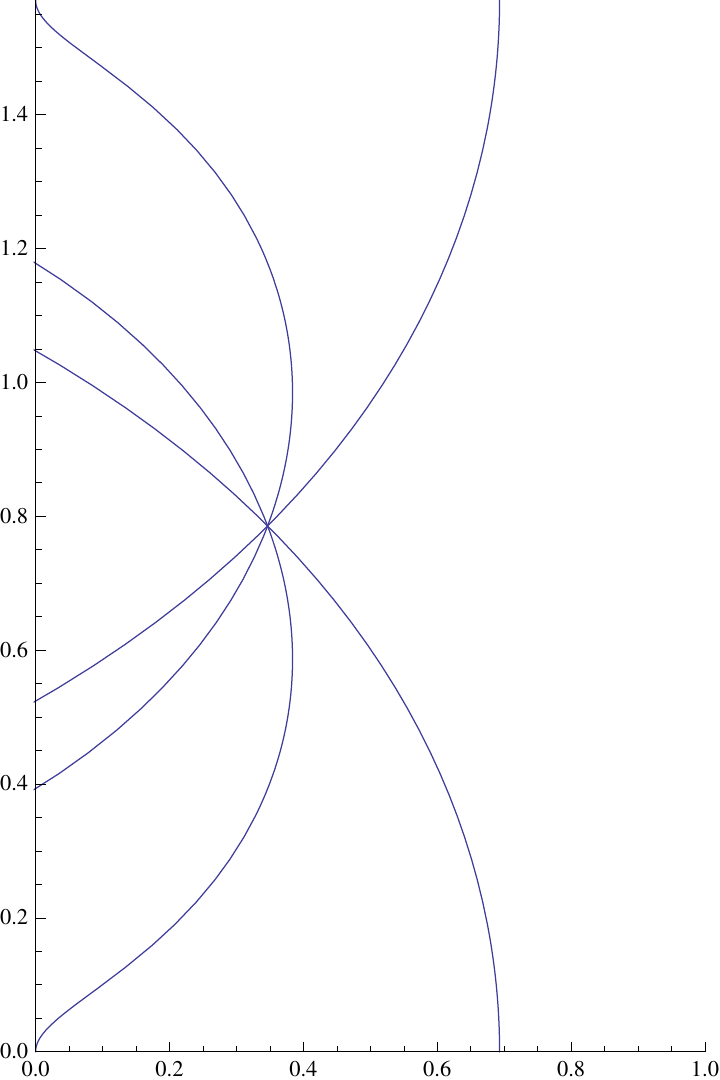} 
\caption{Plots of the curves $\mathcal C_1$,
 $\mathcal C_2$, $\mathcal C_3$, $\mathcal C_4$ and their union $\mathcal C$.}
 \label{fig:c-curves}
\end{figure}

We can now formally state the approximation result, which provides an
explanation for the Observations.
\begin{thm}
  \label{thm:bor}  
The sets $\widehat{\mathcal S}_{X_b}$ and $\mathcal C$ are close in
the      Hausdorff metric~$\ddh$ on a large part of the strip $0<\Re(s)<\log 2$.  More
      precisely, there exists $\varkappa  > 0$  such that 
      $$
      \ddh( \widehat {\mathcal S}_{X_b} \cap \widehat{\mathcal
      R}(e^{\varkappa b}) ,\mathcal C  \cap \widehat{\mathcal
      R}(e^{\varkappa b}) )=O\Bigl(\frac{1}{\sqrt b}\Bigr), 
       \mbox{ as } b \to +\infty.
      $$ 
\end{thm}
  The theorem implies that every rescaled zero $s\in\widehat {\mathcal S}_{X_b}
  \cap \widehat{\mathcal R}(e^{\varkappa b})$ belongs to a neighbourhood of
  $\mathcal C$ which is shrinking as $b \to \infty$. On the other hand, the rescaled 
  zeros are so close, that the union of their shrinking neighbourhoods contains $\mathcal
  C$. 

\begin{rem}
  Although an explicit estimate on the distance between zeros and curves 
  $O\Bigl(\frac{1}{\sqrt b}\Bigr)$ is a bonus, 
  the most significant feature of this result is that the height $e^{\varkappa b}$ of
  the rescaled strip $\widehat{\mathcal R}(e^{\varkappa b})$ is larger than the
  period of the curves $\mathcal C$,
  and it corresponds to a part of the original strip of the height
  $e^{(1+\varkappa)b}$ which allows us to find a set of $\varepsilon$-translations. 
\end{rem}

\begin{figure}[h]
\begin{tabular}{cc}
\includegraphics[scale=0.5,angle=0 ]{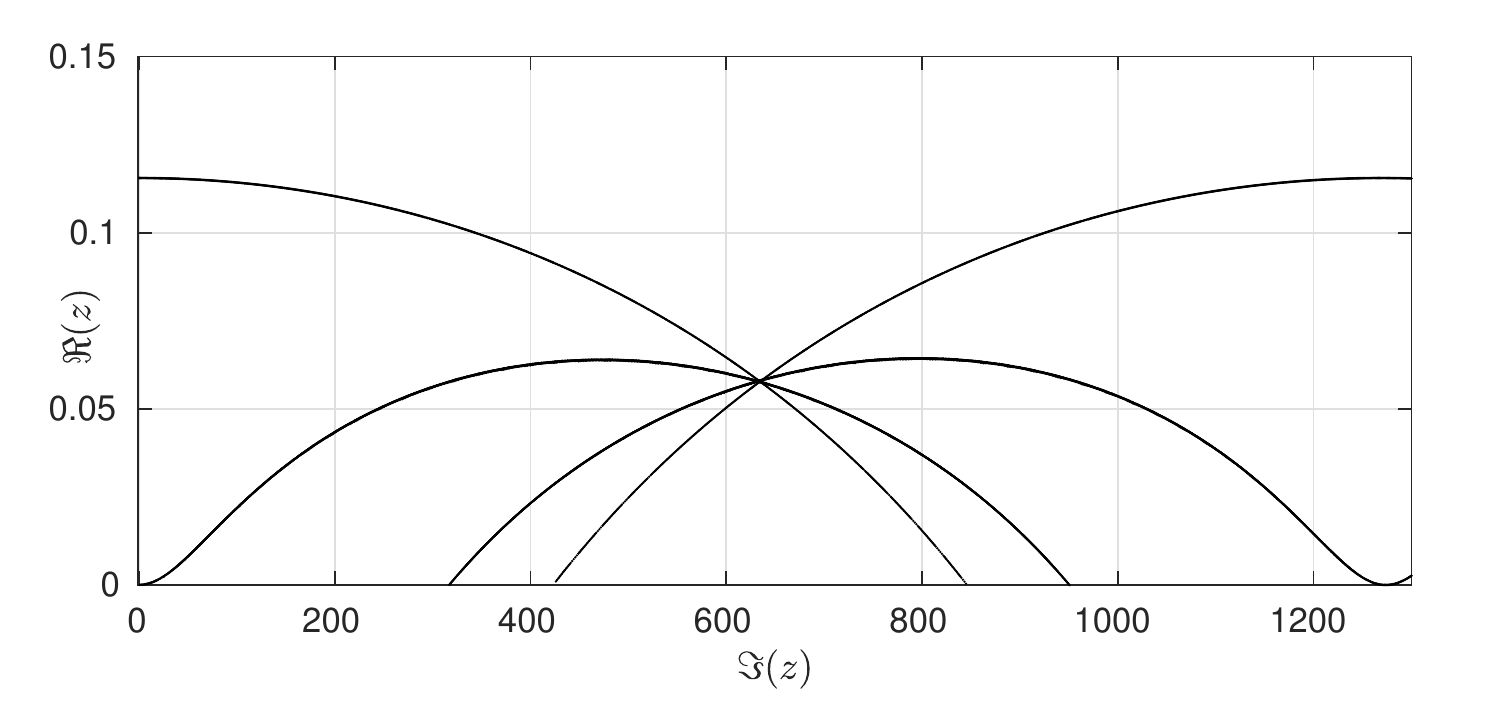} &
\includegraphics[scale=0.5,angle=0 ]{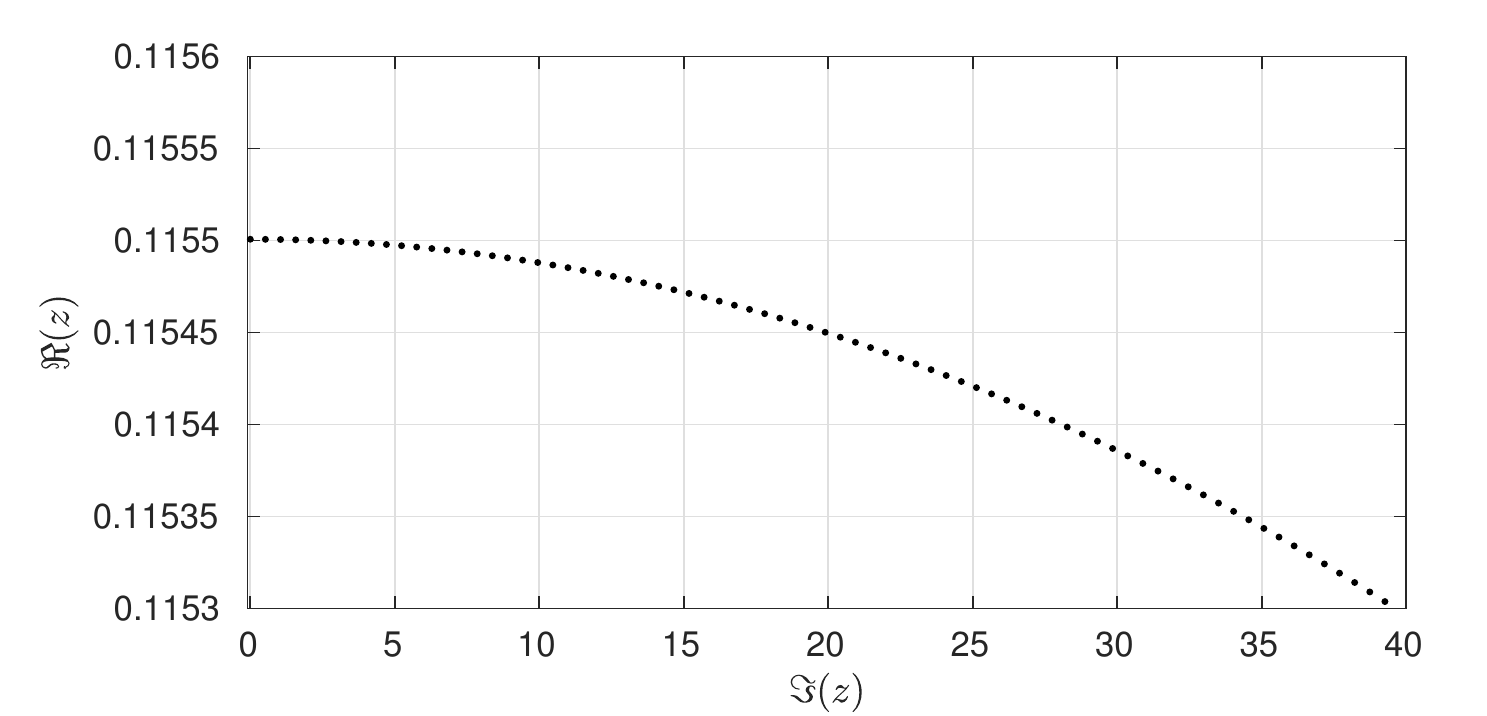} \\
(a) & (b) \\
\end{tabular}
\caption[Four curves]{(a) The zeros of the $Z_{14}$ approximating $Z_{X_6}$;  (b) A zoomed
version in a neighbourhood of $\delta$ showing indvidual zeros. The distance
between imaginary parts of consecutive zeros is approximately $\frac{\pi}{6}$.}
\label{fig:uni}
\end{figure}

\begin{rem}
A similar analysis can be carried out for a punctured torus and for less
symmetric surfaces. However, in most cases the modul\footnote{The minimal
subalgebra of $\mathbb R$ containing all multipliers.} of the exponential sum,
approximating the zeta function with suitable accuracy, has more generators, 
the curves containing zero set can be defined only implicitly, and the set
of $\varepsilon$-translations doesn't have such a simple form. 
\end{rem}

There is an interesting conjecture due to Jakobson and Naud which we include
here for context (see \cite{JN12}, Conjecture 1.1, p.354).

\begin{conjecture}[Jakobson--Naud]
  There are only finitely many zeros in the half plane $\Re(s) >
  \frac{\delta}{2}$ and this is the largest half-plane with this property (i.e.,
  for any $\varepsilon > 0$ there are infinitely many zeros in the half-plane
  $\Re(s) > \frac{\delta}{2} - \varepsilon$).
\end{conjecture}

We also note that there are interesting  empirical investigations on the
(Ruelle) zeta function in the case of Sinai billiards~\cite{SL02}.

This project has had a long gestation period, having begun after  the
first author heard the original empirical results of D.~Borthwick presented at a
conference on Quantum Chaos in Roscoff in June of 2013.  We
are grateful to him for sharing his original Matlab code with us. We are
grateful to F.~Bykov for writing a new program. 
A preliminary announcement of these results was made by the first author 
at the conference ``Spectral problems for hyperbolic dynamical systems'' held in
Bordeaux in May of 2014.

\section[Selberg Zeta Function]{The zeta function and closed geodesics}
\label{s:zg}
We begin this section with the some basic  background on the surface $X_b$. 
Any surface with  constant curvature $\kappa=-1$ has universal covering space the unit 
disc $\mathbb D^2 = \{z = x+iy \colon |z| < 1\}$ equipped with the usual 
 Poincar\'e metric $ds^2 = 4\frac{dx^2 + dy^2}{(1- x^2 - y^2)^2}$.
 In particular, $(\mathbb D^2, ds)$ is a simply connected Riemann surface with constant negative
curvature~$-1$ and $X_b$ can be viewed as a double oriented cover of a quotient $\mathbb
D^2/\Gamma$ by a discrete subgroup of isometries $\Gamma$ (i.e., a Fuchsian
group).  

There are many different choices of generators for $\Gamma$.  
Because of the natural symmetries of $X_b$, it is convenient  to choose a
presentation of the associated Fuchsian group in terms of three reflections (as in \cite{M98}, for example).
More precisely, we  can fix a value $0<\theta\le\frac{\pi}{3}$ and consider the
Fuchsian group $\Gamma = \Gamma_\theta  := \langle R_1, R_2,
R_3\rangle$ generated by reflections $R_1, R_2, R_3$ in  three
disjoint equidistant geodesics $\beta_1,\beta_2,\beta_3$, with end
points~$e^{(\frac{2\pi}{3}j\pm\theta)i}\in \partial\bbD^2$,
$j=1,2,3$, respectively (cf. Figure \ref{fundamental1}). 
\begin{figure}[h]
 \centering
  \includegraphics[scale=1]{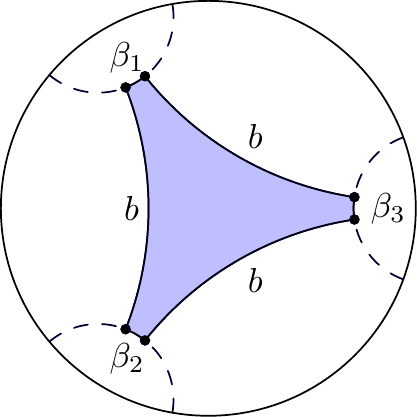}
\caption[Fundamental domain]{Three geodesics of reflection in Poincar\'e disk with pairwise
distance~$b$. 
The compact connected subset $K_b$  of $X_b$ 
bounded by the three simple closed geodesics 
is a double cover for  the hyperbolic hexagon illustrated.}
  \label{fundamental1}
\end{figure}
 Although the three individual generators are 
 orientation reversing the resulting quotient  surface $X_b = \mathbb D^2/\Gamma$ is an oriented
  infinite area surface. Both the infinite area surface and the corresponding
  compact surface $K_b$ with boundary consisting of three geodesics of the
  length $2b$, share the same closed geodesics, and thus have the same zeta function.
 Furthermore, the compact surface  $K_b$ (sometimes known as a pair of pants) is
 precisely the  double cover for the  hyperbolic hexagon in
 Figure~\ref{fundamental1},  (see Expos\'e 3,~\cite{FLP}) 
\footnote{This is because the  reflections $R_1, R_2, R_3$ reverse orientation. 
Composing pairs of reflections $R_1R_2$ and $R_1R_3$, say, would give 
orientation preserving boundary identifications on the fundamental 
domain of $X_b$ and generate a free group isomorphic to
$\pi_1(X_b)$ as in \cite{B14}, but at the expense of the natural symmetry.}.
The values $b$ and $\alpha$ can be related using a simple hyperbolic geometry
calculation.

Moreover, 
since $X_b$ is a negatively curved surface it is a classical  result of E. Cartan that we can   associate a unique closed geodesic~$\gamma$ to
each conjugacy class in the fundamental group~$\pi_1(X_b)$
(see Theorem 2.2 of Chapter of \cite{Carmo}; or Chapter 6 of \cite{Be}). 
\begin{notation}
We shall denote by $\ell(\gamma)$ length of $\gamma$ in the
hyperbolic metric of the surface.  We denote by 
$\omega(\gamma)\in 2\mathbb N$ the {\it word
length}, that is  the (even) number of generators $R_1, R_2, R_3$
required to represent the   conjugacy class corresponding to~$\gamma$ in~$\Gamma$. 
Geometrically,  the word length corresponds to the period 
of the associated cutting sequence, i.e, the sequence of reflections  corresponding to the
sides of the hexagons crossed consecutively  by the  geodesic~\cite{series}
and whose period corresponds to the total number of edges crossed.
\end{notation}
\begin{rem}
  \label{even:rem}  
Therefore, since~$K_b$ is a double cover for a hyperbolic hexagon
(cf. Figure \ref{fundamental1})  every closed geodesic must traverse each  of the two copies 
of the  hexagon  consecutively the same number of times and thus  the  word length is  
necessarily even.
\end{rem}


\subsection{Analytic approximation}
Our starting point for understanding the properties of 
the zeta function is the following important result of D.\,Ruelle from~$1976$,
stating that the infinite product~\eqref{eq:z1} defines an analytic function.

\begin{thm}[after Ruelle]
  \label{thm:ruelle1}
  Let $\delta>0$ be the largest real zero for $Z_{X_b}$. 
  In the notation introduced above, the infinite product~\eqref{eq:z1} converges to a
  non-zero analytic function  for~$\Re(s) > \delta$ and extends as an analytic function to~$\mathbb C$.
\end{thm}
\begin{proof}
  The convergence follows from more general results on Axiom~$A$ flows which we can
  apply to the geodesic flow restricted to the recurrent part~\cite{PP90}.  
  The analyticity follows from applying ideas from the  work of
  Ruelle~\cite{R76}, see also~\cite{PP90} for more details.
\end{proof}

The Selberg zeta function, initially defined by~\eqref{eq:z1} in terms of
lengths of infinitely many closed geodesics, is not an object which can be easily 
computed numerically. A computer can only deal with a finite set of geodesics. 
Naturally, a fundamental question arises: how to choose the geodesics to obtain 
a good approximation to the infinite product? 

The approach of Ruelle provides an approximation of 
the zeta function $Z_{X_b}$ by finite exponential sums, whose multipliers and
coefficients depend on all geodesics corresponding to the word length less
than~$n$, which make it suitable for numerical experiments. 

We begin by considering a more general function in two complex variables
\begin{equation}
  Z_{X_b}(s,z)= \prod_{n=0}^\infty\prod_{\stackrel{\gamma = \mbox{\scriptsize
  primitive}}{\mbox {\scriptsize closed   geodesic}} } \left(1-
z^{\omega(\gamma)}e^{-(s+n)\ell(\gamma)}\right),
\label{eq3.1}
\end{equation}
which converges for~$|z|$ sufficiently small and~$\Re(s)$ sufficiently large.  
Taking $z=1$ we recover the original zeta function $Z_{X_b}(s) = Z_{X_b}(s,1)$.
We follow Ruelle~\cite{R76} in re-writing the infinite product  as a series
\begin{equation}
  Z_{X_b}(s,z) = 1 + \sum\limits_{n=1}^\infty a_n(s) z^n, \label{eq3.2}
\end{equation}
by taking the Taylor expansion in~$z$ about~$0$.  It is then easy to see that~$a_n(s)$
is defined in terms of  finitely many closed geodesics with word lengths at most~$n$.   
In fact this series converges to a bianalytic function for both  $z,s \in
\mathbb C$ as shown by the next Theorem (which is stronger than
Theorem~\ref{thm:ruelle1}).

\begin{thm}[Ruelle~\cite{R76}]
 \label{thm:ruelle2} Using the notation introduced above,
 there exists $C = C(s) > 0$ and $0 < \alpha < 1$ such that $|a_n| \leq C \alpha^{n^2}$
 and thus the  series~\eqref{eq3.2} converges. In particular, we can deduce that
 $Z_{X_b}(s,z)$ is analytic in both variables.   
\end{thm}

Theorem~\ref{thm:ruelle1} is a corollary of the stonger
Theorem~\ref{thm:ruelle2}. We next make an easy observation.
\begin{lemma}
The odd  coefficients vanish i.e., $a_1= a_3 = \cdots  =0$.  
\end{lemma}
\begin{proof}
It was observed in Remark~\ref{even:rem} that $\omega(\gamma)
\in 2 \mathbb N$ in the infinite product~\eqref{eq3.1} and thus only the even  terms 
$a_2, a_4, \cdots $ can be non-zero in~\eqref{eq3.2}.
\end{proof}

We can formally rewrite the zeta function~\eqref{eq3.1} as
\begin{equation}
  \label{eq:zeta2}
Z_{X_b}(s,z) = \exp\left(- \sum\limits_{m=1}^\infty \frac{z^m}{m}  \sum\limits_{\omega(\gamma)=m}
\frac{e^{-s\ell(\gamma)}}{1- e^{-\ell(\gamma)}} \right) 
=\exp\left(- \sum\limits_{m=1}^\infty \frac{b_m(s)}{m} z^m  \right)
\end{equation}
where
\begin{equation}
  \label{eq:zcoefs}
b_m(s) = 
\begin{cases}
\sum\limits\limits_{\omega(\gamma)=m} \frac{e^{-s\ell(\gamma)}}{1- e^{-\ell(\gamma)}} 
& \mbox{ if } m \mbox{ is even; }\cr
0 &  \mbox{ if } m \mbox{ is odd; }
\end{cases}
\end{equation}
and then expand the exponential as a power series and obtain coefficients~$a_n$
comparing~\eqref{eq3.2} with~\eqref{eq:zeta2},
since these series converge provided $\Re(s)>1$. In particular, we can easily
check that the first three non-zero terms are:
$$
\begin{aligned}
a_2(s) &=  -\frac{1}{2} b_2(s);\cr
a_4(s) &=  - \frac{1}{4}\left( b_4(s) -  \frac{b_2^2(s)}{2}\right); \cr
a_6(s)  &= - \frac{1}{6} \left( 
b_6(s) + \frac14b_2(s) b_4(s) - \frac{b_2^3(s)}{8}\right);
\end{aligned}
$$
and in general, 
$$
a_n(s) = -\frac{1}{n}\sum_{j=0}^{n-2} a_j(s) b_{n-j}(s).
$$
Combining the latter with~\eqref{eq:zcoefs}, we deduce that each of the
coefficients $a_n(s)$ is an exponential sum whose multipliers are sums of the
lengths of several closed geodesics of the total word length $n$.
\begin{rem}
  \label{lem:gcount}
  It is important to know the number of closed geodesics of word length $2n$
  for practical applications. We can count them by counting periodic cutting
  sequences. It's easy to see that the number of cutting sequences of period
  $2n$, associated to closed geodesics, i.e. satisfying the additional condition
  $\xi_1 \ne \xi_{2n}$, 
  satisfies the reccurence relation $p_n = 3\cdot 4^{n-1} + p_{n-1}$, $p_1=6$.  
  Hence there are exactly $p_n=4^{n}+2$ closed geodesics of the word length
  $\omega(\gamma)=2n$.
\end{rem} 
Straightforward approximations to $Z_{X_b}$ obtained by evaluating the
series~\eqref{eq3.2} at $z=1$ allow one to compute zeros one-by-one numerically
using the Newton method. 
In order to find explicit curves, we need a further simplification. 
Another approximation to $Z_{X_b}$ can be obtained by replacing the lengths of 
closed geodesics $\ell(\gamma)$ in~\eqref{eq:z1} by suitable close
approximations. 
This will result in a function with zero set on $\mathcal C$.
Afterwards, we shall show that both approximations are sufficiently close to
each other. 
\begin{rem}
Applying general results~cf.~\cite{L80} Appendix 1, on zeros of exponential sums we
also deduce that the zero set of a finite sum $1+\sum\limits_{k=1}^n a_k(s)$ belongs 
to a strip parallel to the imaginary axis, and that the difference between 
imaginary parts of consecutive zeros is approximately $\frac{\pi}{b}$. This was
proved by Weich~\cite{W15} for small values of $s\in \mathcal S_{X_b}$,
without relying on the theory of almost periodic functions, but instead using
``symmetry reduction'' first suggested in~\cite{CE89}.
\end{rem}

\subsection{Geometric approximation}
\label{ss:exp}

The following simple trick allows us to reduce the problem of locating zeros of
the double infinite
product~\eqref{eq:z1} to a problem of locating zeros of a single infinite product. 
Namely, consider the  related Ruelle zeta  function $\zeta(s)$
defined by
\begin{equation}
  \label{eq:zeta-ruelle}
\zeta(s)
= \prod_{\stackrel{ \gamma = \mbox{ \scriptsize primitive}}{\mbox{ \scriptsize
closed geodesic }}}  \left(1-
e^{-s\ell(\gamma)}\right)^{-1}
 = \frac{Z_{X_b}(s+1)}{Z_{X_b}(s)},
\end{equation}
where $\gamma$ again denotes a primitive closed geodesic of length $\ell(\gamma)$.
Since $Z_{X_b}$ is real analytic and non-zero for $\Re(s)>1$, it has 
 poles corresponding to the zeros $Z_{X_b}(s)$ in the strip $0 < \Re(s) < \delta$. 
The function~$\zeta$ is the exact form of the zeta function studied by Ruelle
\cite{R76} and is better suited to geometric approximation. 

We will be following the approach in~\cite{PP90}, a similar argument can be
found in~\cite{T11}.

Let $\Sigma$ be the set of infinite cutting sequences corresponding to 
 geodesics on the $3$-funnelled surface $X_b$. We consider a function 
  \begin{equation}
  \tilde r_n  \colon \Sigma \to \mathbb R \qquad \tilde r_n(\xi) =
  \ell(\gamma_{\xi_{[n/2]} \xi_{[n/2]+1}}), \label{r-tilde:eq}
\end{equation}
where $\gamma_\xi$ is a closed geodesic corresponding to the cutting sequence
$\xi$ and $\gamma_{\xi_{[n/2]} \xi_{[n/2]+1}}$ is the shortest geodesic segment between
the intersections with the sides corresponding to $R_{\xi_{[n/2]}}$ and
$R_{\xi_{[n/2]+1}}$ which lies on a longer segment passing consecutively through the
sides corresponding to the reflections $R_{\xi_1},
\ldots ,R_{\xi_{n+1}}$. 
In particular, any geodesic is uniquely defined by its cutting sequence and
therefore if the word length $\omega(\gamma_\xi)\le 2n$ we have
\begin{equation}
\ell(\gamma_\xi) = \sum_{j=0}^{\omega(\gamma)-1} \tilde r_n(\sigma^j \xi); \mbox{
and }
\label{rsum:eq}
\end{equation}
in general, we have that 
\begin{equation}
\ell(\gamma_\xi) = \lim_{n \to \infty} \sum_{j=0}^{\omega(\gamma)-1} \tilde r_n
(\sigma^j \xi),\label{rsum-lim:eq}
\end{equation}
where $\sigma \colon \Sigma\to\Sigma$ is a shift given by $\sigma(\xi_n) =
(\xi_{n+1})$.

We know that periodic sequences in $\Sigma$ are periodic orbits of a subshift of finite
type (cf.~\cite{P83}, pp.~11--12). We denote by~$A$ the transition matrix of the subshift corresponding to the
encoding by subsequences of length~$n$. Let $\xi^1,\ldots,\xi^N$ be all
subsequences of the sequences in $\Sigma$ of the length~$n$. (In the case
of~$X_b$ we have that $N=3\cdot2^{n-1}$.) We define an $N \times N$
matrix
\begin{equation}
A^n_{i,j} = 
\begin{cases}
  1, & \mbox{ if } \xi^i_{k+1} = \xi^j_{k}; \mbox{ for } k=1,\ldots,n-1 \\
  0, & \mbox{ otherwise.}
\end{cases}
\label{mata-def:eq}
\end{equation}
We use the transition matrix $A^n$ for the subshift to define a family of
matrices $A(s)$ whose elements depend on the length of geodesics segments corresponding to 
transitions. Namely, we introduce a matrix-valued complex function $A(s)$ by 
\begin{equation}
  \label{as-def:eq}
 A \colon \mathbb C \to Mat(N,N) \qquad A_{i,j}(s) = \exp(-s \tilde
 r_n(\xi))\cdot A^n_{i,j},  
\end{equation}
where $\xi = \xi_1^i \ldots \xi_n^i \xi_n^j$. 
Note that $A(s)$ depends on $n$ but we omit this.

The following fact is very useful and relatively well-known
cf.~\cite{P83},\cite{PP90},\cite{T11}. 
\begin{lemma}
  \label{PPT:lem}
  Using the notation introduced above the following
  equality holds
  \begin{equation}
    \prod_{\stackrel{\gamma = \mbox{ \scriptsize primitive } }{\mbox{\scriptsize closed
    geodesic}}} \left(1-e^{-s \ell(\gamma)}\right)^2 = \lim_{n \to \infty}
     \det\left(I_N - A^2(s)\right);
  \end{equation}
  where $I_N$ is the $N\times N$ identity matrix. 
\end{lemma}
\begin{proof}
  Observe that on the domain of convergence the right hand side 
  \begin{align*}
  \log \left( \det\left(I_N - A^2(s)\right) \right) &= - \sum_{j=1}^\infty
  \frac{1}{j} \tr   \left(A^{2j}(s)\right) = -\sum_{j=1}^\infty \frac{1}{j}
  \sum_{\sigma^{2j} \xi = \xi} e^{- s (\tilde r_n(\xi)+\tilde
  r_n(\sigma \xi)+\cdots+\tilde r_n(\sigma^{2j-1} \xi))}   \\
  \intertext{ we may rewrite the latter term as }
  &=-\sum_{j=1}^\infty \sum_{m=1}^\infty \frac{1}{jm} 
  \sum_{\stackrel{\sigma^{2j} \xi = \xi}{\xi \mbox{ \scriptsize primitive}} } e^{- s m (\tilde r_n(\xi)+\tilde
  r_n(\sigma \xi)+\cdots+\tilde r_n(\sigma^{2j-1} \xi))},  \\
  \intertext{   where the inner summation is taken over primitive fixed points for
  $\sigma^{2j}$ }
  &=-2\sum_{j=1}^\infty \sum_{m=1}^\infty \frac{1}{m} 
  \sum_{\stackrel{\{\xi, \sigma\xi, \sigma^2\xi,\cdots,\sigma^{2j-1}\xi \}}{
  \mbox{ \scriptsize primitive}}} e^{- s m (\tilde r_n(\xi)+\tilde
  r_n(\sigma \xi)+\cdots+\tilde r_n(\sigma^{2j-1} \xi))}, 
  \intertext{   where the inner summation is taken over primitive closed orbits
  of period $2j$} 
  & =-2\sum_{j=1}^\infty  
  \sum_{\stackrel{\{\xi, \sigma\xi, \sigma^2\xi,\cdots,\sigma^{2j-1}\xi \}}{
  \mbox{ \scriptsize  primitive}} } \log\left( 1 - e^{- s (\tilde r_n(\xi)+\tilde 
  r_n(\sigma \xi)+\cdots+\tilde r_n(\sigma^{2j-1} \xi))}\right), \\
  & =-2 
 \log  \prod_{j=1}^\infty \prod_{\stackrel{\{\xi, \sigma\xi, \sigma^2\xi,\cdots,\sigma^{2j-1}\xi \}}{
  \mbox{ \scriptsize primitive}}}  \left( 1 - e^{- s (\tilde r_n(\xi)+\tilde 
  r_n(\sigma \xi)+\cdots+\tilde r_n(\sigma^{2j-1} \xi))}\right).
\end{align*}
Using~\eqref{rsum-lim:eq} and swapping the limits, we get the result, since
there is a bijection between oriented primitive closed geodesics and primitive
periodic cutting sequences.

\end{proof}
The last lemma establishes a connection between the Ruelle zeta function and
the determinant of certain matrices. It follows from
definition~\eqref{as-def:eq} that $\det(I_N - A^2(s))$ is a finite exponential
sum as a function of~$s$.  
It turns out that in the particular case of~$X_b$, the zero
set of $\det(I_N - A^2(s))$ for $N=6$ is easy to describe and that for $b$ large this exponential
sum is close to $1+\sum\limits_{j=1}^{\infty} a_n(s)$, where $a_n(s)$ are the coefficients
in~\eqref{eq3.2}. 

\subsection{Computing a geometric approximation }
We apply the method described in the previous paragraph in a simple case $n=2$. 
In other words, we will be computing the length of segments of closed geodsics
taking into account only $n+1=3$ consecutive elements of the cutting sequences,
and computing the total length by summing up the length of the pieces. We will
show afterwards that this approximation is good enough.

In order to compute $\tilde r_2$ as defined by~\eqref{r-tilde:eq}, we observe the following simple fact:
\begin{lemma}
  \label{lem:eb}
Consider a regular hyperbolic hexagon whose even sides are of length $b>1$ and whose odd sides are 
of length $\varepsilon_b$ then 
$$
\varepsilon_b = 2 e^{-b/2} + e^{-b} + O(e^{-3b/2})\quad \mbox{ as }
b\to+\infty.
$$
\end{lemma}
\begin{proof}
  We recall that (cf.~\cite{B83},  Theorem 7.19.2)
$$
\cosh \varepsilon_b (\sinh b)^2 = \cosh b + (\cosh b)^2 . 
$$
The result follows by expanding the both sides in $e^{-b}$ as $b \to +\infty$ and 
comparing the expansions.
\end{proof}
 Now we are ready to describe $\tilde r_2 \colon \Sigma \to \mathbb R^+$. 
\begin{lemma}
  \label{lem:lgam}
Assume that $\xi \in \Sigma$ is a cutting sequence of period~$2n$ given by 
$$
\xi = \xi_1 \xi_2 \xi_3 \cdots \xi_{2n-1} \xi_{2n} \xi_1 \cdots,
$$
where $\xi_k \neq \xi_{k+1}$ for $1 \leq k \leq 2n$ and $\xi_{2n} \neq \xi_1 $.  Then 
$\tilde r_2(\xi) = b + c(\xi) e^{-b} + O(e^{-2b})$ as $b\to + \infty$, where 
$$
c(\xi) = 
\begin{cases}
  0, & \mbox{ if } \xi_{1}=\xi_{3}; \\ 
  1, & \mbox{ otherwise }. 
\end{cases}
$$
\end{lemma}

\begin{proof}
  We need to estimate the length of the shortest of the segments of closed geodesics
  passing through $\beta_{\xi_1}$, $\beta_{\xi_2}$, and $\beta_{\xi_3}$ enclosed
  between the intersections with $\beta_{\xi_1}$ and $\beta_{\xi_2}$.

  Let us first consider the case when $\xi_1 = \xi_3$, see geodesic
  segment~$\gamma_{1 3 1}$ in Figure~\ref{fig:hexag} for example.  
  It is evident that the shortest geodesic whose cutting sequence has a
  subsequence $\xi_1 \xi_2 \xi_1$ is the boundary one. Therefore, we conclude
  that   $\tilde r_2(\xi) = b$. 

\begin{figure}[h]
\centerline{
\includegraphics[scale=0.8]{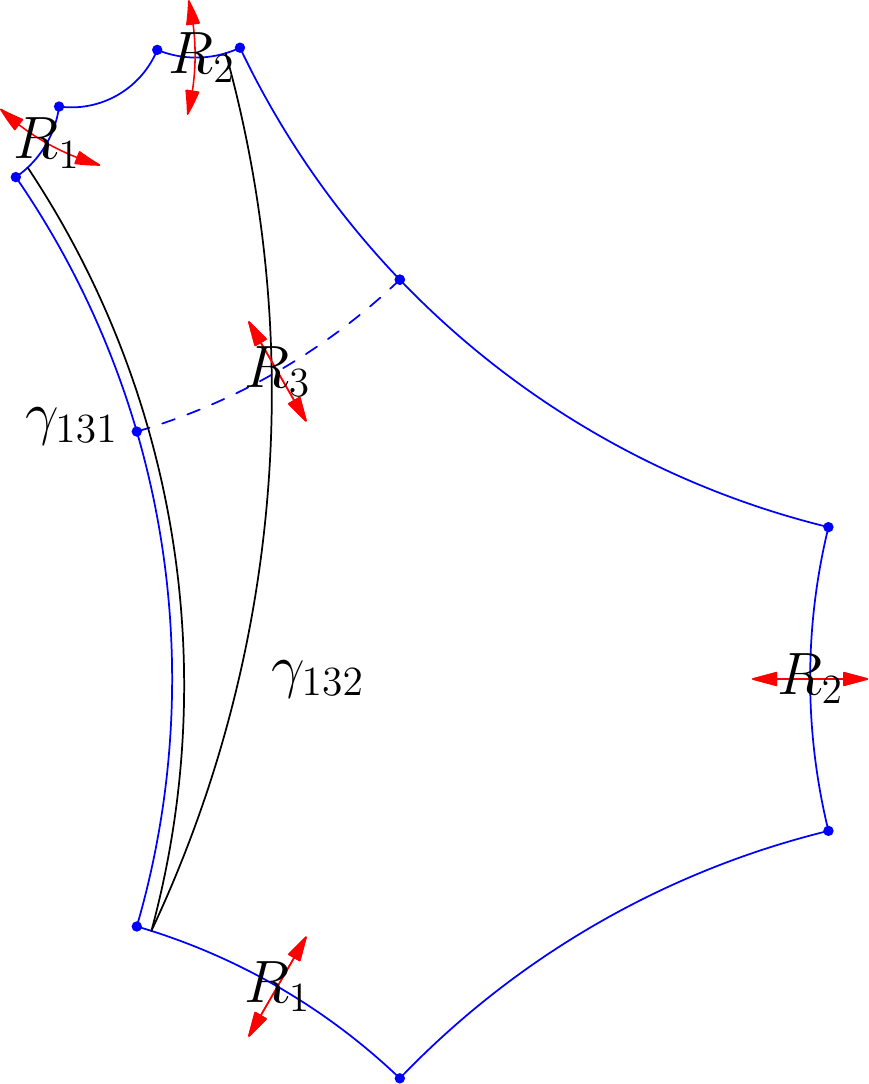}
}
\caption[Geodesics and cutting sequences]{A geodesic
segment~$\gamma_{131}$, running along the side of two hexagons when
$R_{\xi_{1}} = R_{\xi_{3}}$; and a geodesic segment~$\gamma_{132}$, traversing
two hexagons when $R_{ \xi_{n-1}} \neq  R_{\xi_{n+1}}$.}
\label{fig:hexag}
\end{figure}

  Now we consider the case $\xi_{1} \neq \xi_{3}$, see the geodesic segment
  $\gamma_{1 3 2}$ in Figure~\ref{fig:hexag} for example. 
 
  We may denote the point of intersection of the geodesic $\gamma_\xi$ with the
  side $R_{\xi_k}$ by $A_k$. We may assume that $A_k$ divides the side of
  reflection in proportions $x_k:(1-x_k)$, for $k=1,\ldots,3$. Then the geodesic
  segment $\gamma_{\xi_1\xi_{2}}$ is a side of a right hyperbolic trapezoid with
  one side of the length $b$, and two parallel sides of the length $x_1 \varepsilon_b$
  and $x_{2} \varepsilon_b$. Similarly, the geodesic
  segment $\gamma_{\xi_2\xi_{3}}$ is a side of a right hyperbolic trapezoid with
  one side of the length $b$, and two parallel sides of the length $(1-x_2) \varepsilon_b$
  and $x_3 \varepsilon_b$
 
  Using hyperbolic sine and cosine laws, we deduce the formula for length of the
  fourth side of the trapezoid: 
  \begin{align*}
  f_1(x_1,x_2) \colon &= \ell(\gamma_{\xi_1\xi_2})  \\&=\acosh\left(\cosh b \cdot \cosh( x_1 \varepsilon_b )
  \cdot \cosh\left( x_2 \varepsilon_b \right) - \sinh( x_1\varepsilon_b)
  \cdot \sinh\left( x_2 \varepsilon_b \right)\right).
  \intertext{Similarly, for another segment: }
  f_2(x_2,x_3) \colon &= \ell(\gamma_{\xi_2\xi_3}) = \\&=\acosh\left( \cosh b \cdot
  \cosh( (1-x_2) \varepsilon_b )
  \cdot \cosh\left( x_3 \varepsilon_b \right) - \sinh( (1- x_2)\varepsilon_b)
  \cdot \sinh\left( x_3 \varepsilon_b \right)\right). 
  \end{align*}
  Consider the function $f_\xi(x_1,x_2,x_3) = f_1(x_1,x_2) + f_2(x_2,x_3)$. 
   By definition, $\tilde r_2(\xi) = f_1(x_1^\prime,x_2^\prime)$, where
  $x_1^\prime$ and $x_2^\prime$ are chosen so that 
  \begin{equation}
  f_\xi(x_1^\prime,x_2^\prime,x_3) = \inf_\xi \inf_{x_1,x_2,x_3}
  f_\xi(x_1,x_2,x_3).
  \label{f-xi-def:eq}
  \end{equation}
  Analysing this function we obtain that $x_1^\prime = 0$
  and $x_2^\prime = 0.5$. We see that 
  $$
  f_1(0,0.5) =  \acosh \left(\cosh b \cdot \cosh\left( 0.5 \varepsilon_b
  \right)\right).
  $$
  Finally by a straightforward calculation we can write
\begin{align}
  \tilde r_2(\xi) &= \ln \left( \cosh b \cosh
\frac{\varepsilon_b}{2} +
\sqrt{\cosh^2 b \cosh^2 \frac{\varepsilon_b}{2} -1} \right) 
= \notag \\ & =
\ln \left( 2 \cosh b \cosh \frac{\varepsilon_b}{2} + \frac12 e^{-b} + O(e^{-2b}) \right)
= \notag \\ & = 
\ln ( 2 \cosh b) + 2 \ln \left(\cosh \frac{\varepsilon_b}{2}\right)  + 2\ln \left( 1 + \frac12
e^{-2b} + O(e^{-3b})\right). \label{eq:cgam1}
\end{align}
We have asymptotic expansions 
\begin{equation}
  \label{eq:cgam2}
\ln(2 \cosh b) = b + \ln (1 + e^{-2b}) = b + e^{-2b} + O(e^{-4b})
\end{equation}
and 
\begin{equation}
  \label{eq:cgam3}
\ln \left(\cosh \frac{\varepsilon_b}{2}\right) = \ln\left(\frac{e^{ e^{-b/2} +
O(e^{-b})} + e^{ - e^{-b/2} + O(e^{-b})}}{2} \right) = \ln\left(1 +
\frac{e^{-b}}{2} +
O(e^{-2b})\right). 
\end{equation}
Substituting~\eqref{eq:cgam2} and~\eqref{eq:cgam3} into~\eqref{eq:cgam1}, we
conclude $\tilde r_2(\xi) = b +  e^{-b} + O(e^{-2b})$ as $b\to +\infty$.
This completes the case $\xi_1\ne\xi_3$ and proves the Lemma.
\end{proof}

In Appendix~\ref{example-ap:s} we obtain estimates for the length of closed
geodesics of the word length $2$, $4$, $6$, and $8$; illustrating
Lemma~\ref{lem:lgam} and formula~\eqref{rsum-lim:eq}.

Now we can use Lemma~\ref{lem:lgam} to compute the matrix $A(s)$ defined
by~\eqref{as-def:eq} in the case $n=2$. We write the transition matrix
$A$ defined by~\eqref{mata-def:eq}, corresponding to~$6$ subsequences 
of cutting sequences of length two: $\xi^1=\{1,2\}$; $\xi^2=\{1,3\}$; $\xi^3=\{2,1\}$; $\xi^4=\{2,3\}$;
$\xi^5=\{3,1\}$; $\xi^6 = \{3,2\}$: 
$$
A =  
\left(\begin{matrix}
0& 0& 1& 1& 0& 0 \cr 0& 0& 0& 0& 1& 1 \cr
1& 1& 0& 0& 0& 
 0 \cr
  0& 0& 0& 0& 1& 1 \cr
  1& 1& 0& 0& 0& 0 \cr
  0& 0& 1& 1& 0& 0
\cr
\end{matrix}\right).
$$
Substituting $\tilde r_2(s)$ into~\eqref{as-def:eq}, we get the matrix function:
$$ A(s) = 
\left(\begin{matrix}
  0& 0& e^{-sb} & e^{-s(b +  e^{-b})} & 0 & 0 \cr 
  0& 0& 0& 0& e^{-sb} & e^{-s(b +  e^{-b})} \cr
  e^{-sb} & e^{-s(b + e^{-b})}& 0& 0& 0 & 0 \cr
  0& 0& 0& 0& e^{-s(b + e^{-b})} & e^{-sb} \cr
  e^{-s(b + e^{-b})}& e^{-sb} & 0& 0& 0 & 0 \cr
  0& 0& e^{-s(b +  e^{-b})}& e^{-sb}& 0 & 0 \cr
\end{matrix}\right) + O\left(e^{-2b}\right).
$$
Introducing a shorthand notation $z=e^{- s e^{-b}}$ we rewrite the main
term of $A(s)$ as 
\begin{equation} 
  \label{A2-def:eq}
A(s) = 
e^{-sb} \left(\begin{matrix}
  0& 0& 1 & z & 0 & 0 \cr 
  0& 0& 0& 0& 1 & z \cr
  1 & z & 0& 0& 0 & 0 \cr
  0& 0& 0& 0& z & 1 \cr
  z & 1& 0& 0& 0 & 0 \cr
  0& 0& z & 1 & 0 & 0 \cr
\end{matrix}\right).
\end{equation}

The simplicity of the matrix $A(s)$ allows us to study the zero set of the analytic function
 $$
 g(s) = \det(I_6 - A^2(s)) = \det(I_6 -A(s)) \cdot \det (I_6 + A(s)).
 $$
Since $\exp(-sb)$ doesn't vanish, we see that $g(s)=0$ if and only if
$$
\det(\exp(2 sb) I_6 - \exp(2sb) A^2(s)) =0
$$ 
and this inequality holds true if and only if $\exp(2 s b)$ is an eigenvalue of
$\exp(2sb)\cdot A^2(s)$. 

The eigenvalues of the matrix 
\begin{equation}
B(z) \colon =  \left(\begin{matrix}
  0& 0& 1 & z & 0 & 0 \cr 
  0& 0& 0& 0& 1 & z \cr
  1 & z & 0& 0& 0 & 0 \cr
  0& 0& 0& 0& z & 1 \cr
  z & 1& 0& 0& 0 & 0 \cr
  0& 0& z & 1 & 0 & 0 \cr
\end{matrix}\right)^2
\label{B-mat:eq}
\end{equation}
can be computed explicitly: 
\begin{align*}
\mu_1(z)&=(z - 1)^2 \\
\mu_2(z)&= (z + 1)^2\\
\mu_3(z)&=1 - \frac{z^2}{2} + z\frac{\sqrt{4 - 3z^2}}{2} \\
\mu_4(z)&= 1 - \frac{z^2}{2} - z\frac{\sqrt{4 - 3z^2}}{2}.
\end{align*}

We summarise our finding in the following Lemma. 

\begin{lemma}
  \label{det0:lem}
The determinant $\det(I_6 - A^2(s))$ vanishes if and only if $\exp(2bs) =
\mu_k(\exp\left(-se^{-b}\right))$ for some $k=1,2,3,4$ where $\mu_k$
as defined above. 
\end{lemma}

We will show how this Lemma leads to the construction of the four curves containing
zeros after computing the errors in approximations of the zeta function by
an exponential sum and the exponential sum by  the determinant $\det\left( I_6 -
A^2(s)\right)$. We would like to finish this section with the following remark
on properties of the matrix $B$.

\begin{rem}
  \label{rem:app3}
  Let us consider the matrix~$B(z)$ defined by~\eqref{B-mat:eq}.
Then the coefficients of the polynomial
$$
\frac{1}{2n}\tr (B^{n}(z)) = d_{2n} z^{2n} +  d_{2n-1} z^{2n-1} + \cdots +  d_{1} z + d_0
$$
are given by
$$
d_{k }  = \#\left\{   \gamma_\xi \colon \omega(\gamma_\xi) = 2n \mbox{ and }
\sum_{j=0}^{2n-1} c(\sigma^j \xi) = k
\right\}.
$$
In particular, $\frac{1}{2n} \tr B^{2n}(1)$ is equal to the number of closed
geodesics of the word length~$2n$. 

\end{rem}

\section{Nuclear operators and analytic functions}
\label{s:trop}
Error estimates for approximations of the function~$Z_{X_b}$ are based on the
original approach in~\cite{R76} (and the interpretation in~\cite{P91}). We begin by recalling some
abstract results, essentially due to Grothendieck, on nuclear operators.  We
then complete the section by relating the length of the boundary geodesics 
to the contraction on the boundary corresponding to reflections, generating
the group~$\Gamma_\theta$.

\subsection{Nuclear Operators}
\label{ss:nucop}
The convergence of the series~\eqref{eq3.2} in Theorem~\ref{thm:ruelle2} will follow from
estimates of Ruelle~\cite{R76}, 
after Grothendieck~\cite{G55}. We summarize below the general theory. 

Let $\mathcal B$ be a Banach space.
\begin{definition}
  \label{def:nucOp}
We say that a linear operator $T: \mathcal B \to \mathcal B$ is  nuclear
if there exist for each $n\ge 1$ 
\begin{enumerate}
\item
 $w_n \in \mathcal B$, with $\|w_n\|_{\mathcal B}=1$; 
\item
 $\nu_n \in \mathcal B^*$, with $\|\nu_n\|_{\mathcal B^*}=1$; 
\item 
 $\lambda_n \in \bbR$, with $0 < \lambda < 1$, $C>0$ satisfying $|\lambda_n| \le
 C \lambda^n$
\end{enumerate}
such that
\begin{equation}
 T f= \sum\limits_{n=1}^\infty  \lambda_n w_n\nu_n(f).
 \label{eq:nucOp}
\end{equation}
\end{definition}

\begin{lemma}[after Grothendieck]
  \label{lem:grot}
  A nuclear operator on a Banach space is trace class, and we can write 
  $$
  \det(I-zT) = \exp\Bigl(-\sum_{n=1}^\infty \frac{z^n}{n} \tr T^n\Bigr). 
  $$
  where the infinite series on the right hand side converges in a small disk
  $|z|<\varepsilon$ and the equality between analytic continuations then holds
  on~$\mathbb C$. 
\end{lemma}
As the left hand side is an analytic function we may expand it in a power series
at $z=0$:
$$
\det(I-zT) = 1 + \sum_{n=1}^\infty a_n z^n,
$$
where
\begin{equation}
a_n = \sum\limits_{j_1 < \cdots < j_n}\det\left([\nu_{j_k}(w_{j_l})]_{k,l =
1}^n\right) \lambda_{j_1} \ldots \lambda_{j_n}.
\end{equation}
Applying estimates of Grothendieck and Ruelle, 
we obtain an explicit bound.
\begin{equation}
  \label{eq:anb}
 |a_n| \leq C^n n^{n/2} \lambda^{n(n+1)/2}, 
\end{equation}
where $n^{n/2}$ bounds the supremum
norm of the matrix\footnote{We follow Ruelle in including the term $n^{n/2}$ although
this can be improved upon by looking at Hilbert spaces of analytic functions.
For instance, O.~Bandtlow and O. Jenkinson~\cite{BJ08} have shown 
that we can supress $n^{n/2}$ by working with Hardy spaces, but then
$\lambda$ would be different, too.} $[\nu_{j_k}(w_{j_l})]_{k,l =
1}^n$ and $C\in\mathbb R$ is a constant.

\subsection{Constructing the Banach space}
In computations it will prove more useful to use the equivalent representation
of hyperbolic space by the upper half-plane $\mathbb H^2 = \{x+i y
\colon y > 0\}$ with the metric $ds^2 = \frac{dx^2 + dy^2}{y^2}$. 

\label{ss:constBsp}
Given four points $z_1 < w_1 < w_2 < z_2$ on the boundary $\partial\bbH^2$, we define the cross ratio by
$$
[z_1,w_1,w_2,z_2]= \frac{(z_1-w_2)}{( z_1 - w_1 )} \frac{(w_1-z_2)}{( w_2 - z_2)}.
$$
We recall the following classical formula (cf. \cite{B83} \S 7.23).
\begin{lemma}
  \label{lem:dist}
Let $L_1, L_2$ be two disjoint geodesics in $\bbH^2$ with end points $z_1, z_2$ and $w_1, w_2$.
The distance~$d(L_1,L_2)$ between $L_1$ and $L_2$ satisfies  $[z_1,w_1,w_2,z_2]
= \tanh^2(d(L_1,L_2)/2)$.
\end{lemma}

By assumption, the group $\Gamma=\langle R_1,R_2,R_3\rangle$ is generated by reflections with respect to three disjoint
geodesics, which we denote by~$\beta_1$, $\beta_2$, and $\beta_3$, respectively. Without loss of generality, we
may assume that the geodesic $\beta_j$ has end points $e^{\bigl(\frac{2\pi j}{3}
\pm \theta\bigr) i}$, for $j =1,2,3$ and a small real number~$\theta$.
More precisely, by straightforward calculation using Lemma~\ref{lem:dist} we get
\begin{lemma}
\label{b-lemma}  
Let $\beta_1$ and $\beta_2$ be two disjoint geodesics in $\mathbb D^2$ with end
points $e^{\bigl(\pm \frac{2\pi}{3} \pm \theta\bigr) i}$.
Then $\sin\theta=\frac{1}{2\cosh b}$.
\end{lemma}
\begin{proof}
We can apply Lemma \ref{lem:dist} with 
$z_1 = e^{\bigl( \frac{2\pi}{3} + \theta\bigr) i}$,
$z_2 = e^{\bigl( \frac{2\pi}{3} - \theta\bigr) i}$, 
$w_1 = e^{\bigl(- \frac{2\pi}{3} + \theta\bigr) i}$,
$w_2 = e^{\bigl(- \frac{2\pi}{3} - \theta\bigr) i}$ and $b = d(L_1, L_2)$.
\end{proof}
\begin{rem}
  In the notation and under the hypothesis of the last lemma, we have an asymptotic relation 
  $$
  \theta= 
  \frac12e^{-b}\left(1+e^{-2b}+o(e^{-3b})\right) \mbox{ as }
  b\to\infty.
  $$ 
\end{rem}
To define the Banach space, we fix a small $\varphi<\theta$ and introduce three additional geodesics $\upsilon_j$ with end
points $e^{\bigl(\frac{2\pi j}{3}\pm\varphi\bigr)i}$, $j=1,2,3$. We may consider the disk
$\bbD^2$ as a subset of $\bbC$ and formally extend the geodesics $\upsilon_j$ to circles
$\overline\upsilon_j \subset \bbC$. Furthermore, let $\{U_j\}_{j=1}^3$ be three
 disks in $\bbC$ such that $\partial U_j = \overline\upsilon_j$ cf. Figure~\ref{fig:D}. 

\begin{figure}[h]
\centering
\includegraphics[scale=0.8]{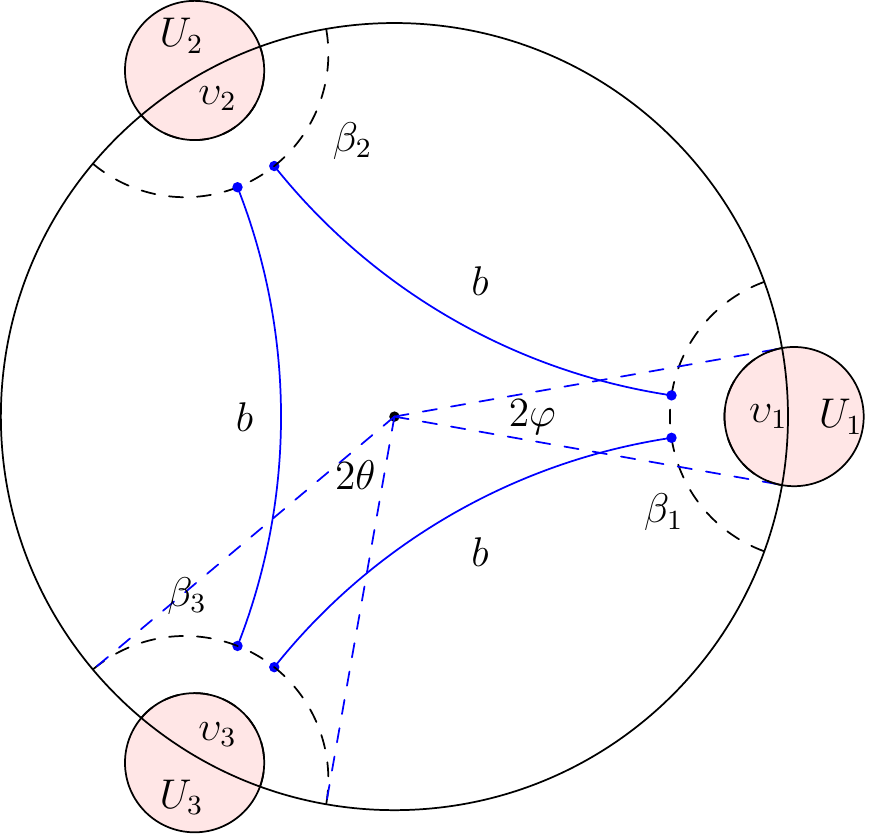}
\caption[The Banach space for Theorem~\ref{thm:main-plus}]{ Three geodesics
$\beta_j$ in~$\bbD^2$ giving rise to the
reflections $R_j$; and three additional geodesics $\upsilon_j$ which are used to
define domain of the analytic functions in $\mathcal B$. The disks $U_j$ are the shaded regions.} 
\label{fig:D}
\end{figure}

The Banach space of bounded analytic functions $f: \sqcup_{j=1}^3 U_j \to \mathbb
C$ on the union $\sqcup_{i=1}^3 U_i$ we denote by~$\mathcal B$. We supply it
with the supremum norm~$\|f\|_{\mathcal B} \colon= \|f\|_\infty$.

\subsection{Transfer operators}
We can now define transfer operators~$\mathcal L_s$, acting on the Banach 
space~$\mathcal B$ of bounded analytic functions on $\sqcup_{j=1}^3 U_j$.

\begin{definition}
For each $s \in \bbC$ we can define 
\begin{equation}
  \label{eq:opLdef}
(\mathcal L_s f)(z) = \sum_{k=1}^3 \chi_{U_k}(z)
\sum\limits_{j\neq k}(R_j'(z))^s f(R_j(z)) \qquad \mbox{  for $f \in \mathcal B$},
\end{equation}
where $\chi_{U_k}$ is the indicator function of~$U_k$. 
\end{definition}
We can apply the general theory of nuclear operators to the transfer operators
by virtue of the following (compare with~\cite{JP02}).
\begin{lemma}
  The operator $\mathcal L_s \colon \mathcal B \to \mathcal B$ is nuclear.
\end{lemma}
\begin{proof}
We observe that the operators $f \mapsto f\circ R_j$ are nuclear and
$R_j^\prime(z)^s$ are analytic~\cite{R76}. Thus $\mathcal L_s$ is  nuclear, too.
\end{proof}
Applying Lemma~\ref{lem:grot} to
$T_s: \mathcal B \times \{\pm 1\} \to \mathcal B \times \{\pm 1\}$ defined by
$T_s(w, \epsilon) = (\mathcal L_s w, -\epsilon)$
we recover Theorem~\ref{thm:ruelle2}, see~\cite{R76} for details.
In particular, by equation \ref{eq:zcoefs} we can write
$$
Z_{X_b}(s,z) = \exp\left(-\sum_{m=1}^\infty
\frac{z^{2m}}{2m} \mbox{\rm tr}\mathcal L_s^{2m} \right) = \det(I - T_s)
$$
where 
$$
\mbox{\rm tr} \mathcal L_s^{2m}
= \sum_{k_1, \cdots, k_{2m}} 
\frac{|(R_{k_1} \cdots R_{k_{2n}})'(x_{k_1, \cdots,
k_{2m}})|^{-s}}{1-\left|(R_{k_1} \cdots R_{k_{2m}})'(x_{k_1, \cdots,
k_{2m}})\right|^{-1}}
$$
where 
$x_{k_1, \cdots, k_{2m}}$ is the 
expanding fixed point for
$R_{k_1} \cdots R_{k_{2m}}$.
This is completely analogous to the approach to transfer operators associated to modular surface in~\cite{Mayer}.

\section{Estimating aproximation errors}
\label{s:result}

\begin{notation}
  We denote a partial sum of the series~\eqref{eq3.2} by $Z_n$ ($n\ge 1$): 
  \begin{equation}
    Z_n(s) \colon = 1 + \sum_{k=1}^n a_k(s).
  \end{equation}
  Note that it is an exponential sum of $\left[\frac{n}{2}\right]$ terms, since
  odd terms vanish $a_{2k+1}=0$.
\end{notation}

Our main approximation result is the following.
\begin{thm}
\label{thm:main-plus}
Let $X_b$ be a symmetric $3$-funnelled surface with defining geodesics of 
length~$2b$. 
Then the finite partial sums $Z_n$ give approximations to $Z_{X_b}$ on the
domain~$\mathcal R(T)$ and the remainder is bounded as follows: 
$\sup_{\mathcal R(T)} |Z_{X_b} - Z_{n}| \le \eta(b,n,T)$ where
$T=T(b)=e^{\varkappa b}$ for some constant $\varkappa>1$ independent of $b$ and
$n$, such that
\begin{enumerate}
  \item for any  $n \ge 14 $ we have $ \eta(b,n,T(b)) = O\Bigl( \frac{1}{\sqrt b} \Bigr) $ as $ b \to \infty$
  \item for any  $b \ge 20$ we have  $ \eta(b,n,T(b)) = O\bigl( e^{-bk_1n^2} \bigr)$
   as $ n \to \infty$
\end{enumerate}
for some $k_1>0$ which is independent of $b$ and $n$.
\end{thm}
For a fixed $b$ this theorem estimates the number of terms $a_n$ needed to uniformly approximate
$Z_{X_b}$ to any given error; at the same time 
for a given~$n$ this theorem estimates the difference between $Z_{X_b}$ and $Z_n$ as
$b \to \infty$ on an exponentially growing domain. 

\begin{rem}
  The constants $\varkappa$
  and  $k_1$ in Theorem~\ref{thm:main-plus} should satisfy the inequality
  $0<k_1<\frac{2-\varkappa}{60}$, although this bound is not sharp. A sharp bound can be obtained using
  the same argument, but the formulae will be more complicated.  
\end{rem}

To illustrate Theorem~\ref{thm:main-plus},
we can fix a surface by choosing the length of boundary
geodesics~$2b$ and plot the zeros for the approximating 
trigonometric polynomials $Z_{2n}=1+a_2 + \cdots + a_{2n}$ for $n=1,2,\cdots,6$. 
For instance, in Figure~\ref{fig:mult} zeros of polynomials approximating $Z_{X_b}$
with $b=5$ are shown. The apparent ``gaps'' in the zeros are due to instability of the
Newton method. 

\begin{figure}[h]
\begin{center}
\begin{tabular}{cc}
\includegraphics[scale=0.500,angle=0 ]{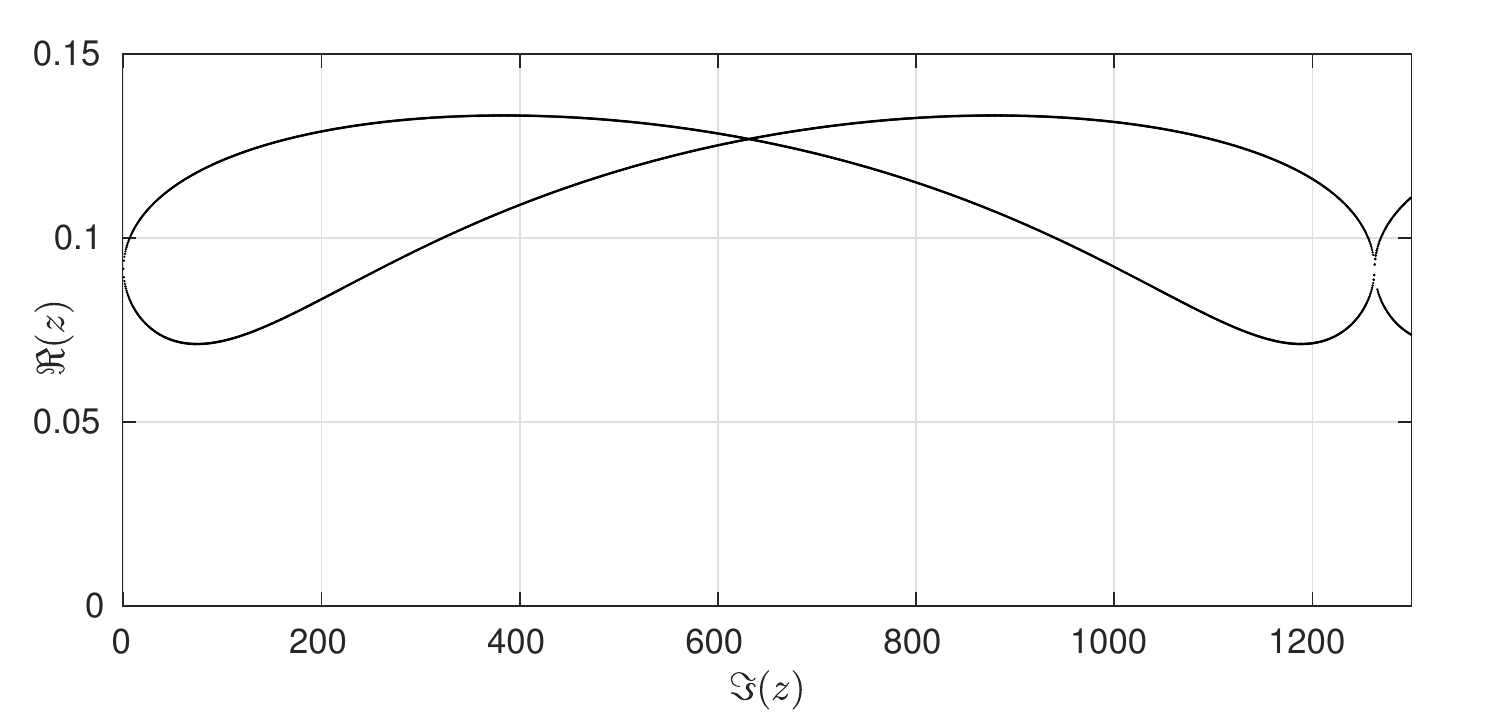} &
\includegraphics[scale=0.500,angle=0 ]{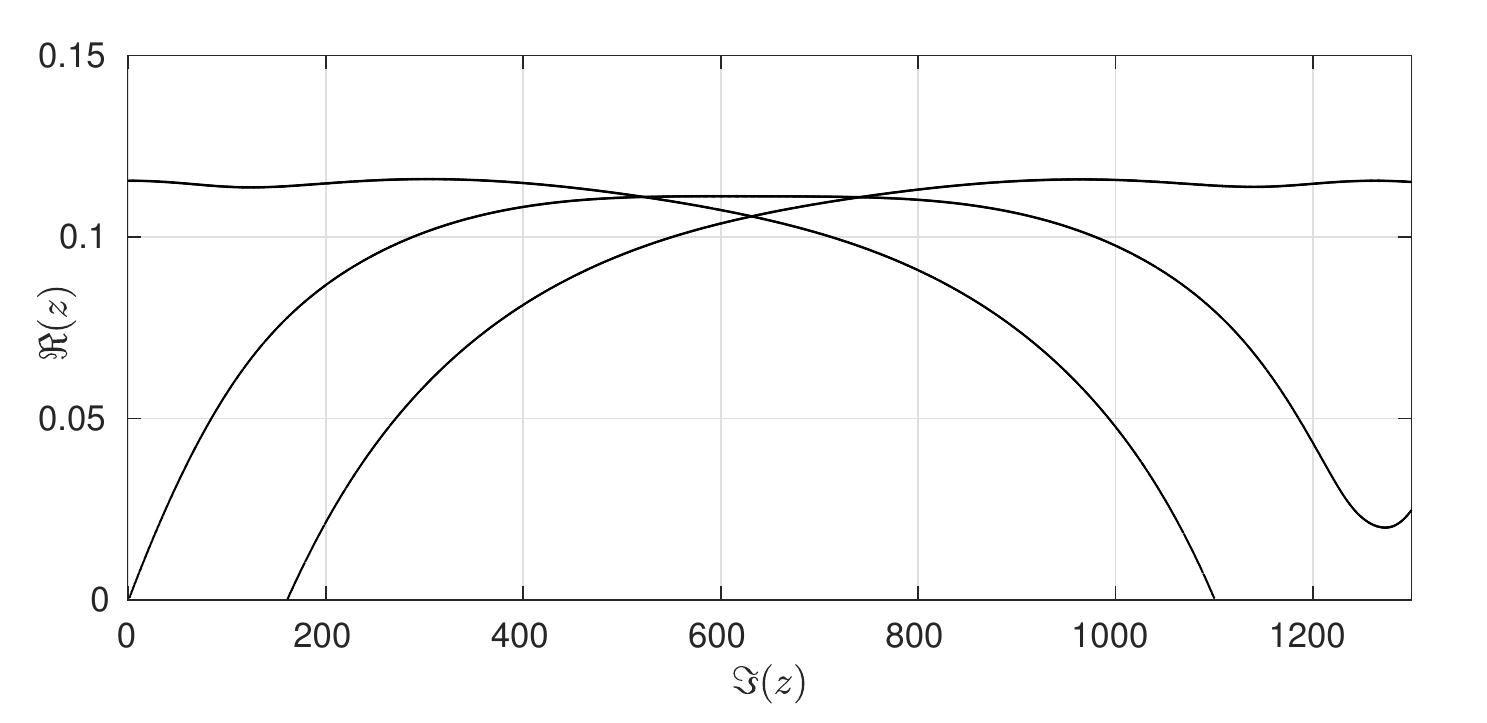} \\
(a) $Z_4(s)$ & 
(b) $Z_6(s)$ \\ 
\includegraphics[scale=0.500,angle=0 ]{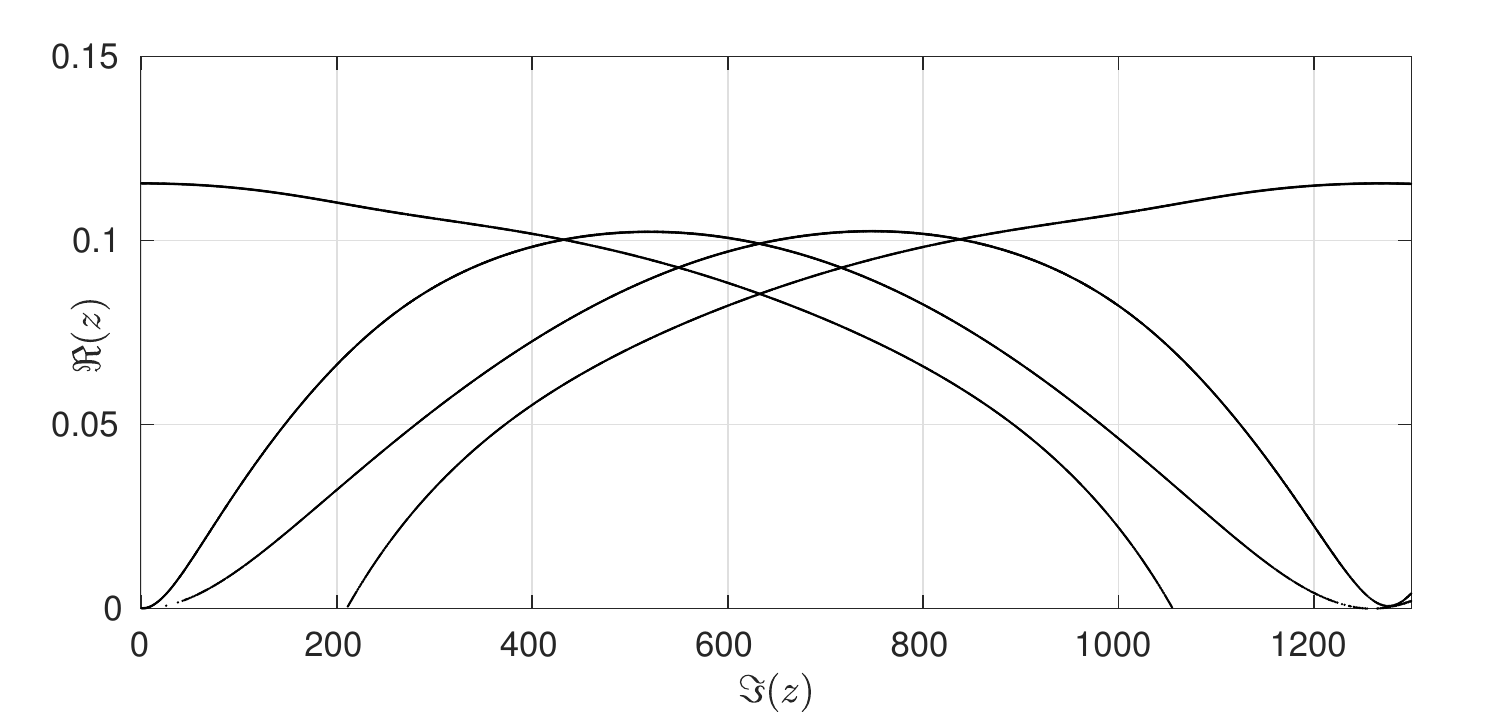} &
\includegraphics[scale=0.500,angle=0 ]{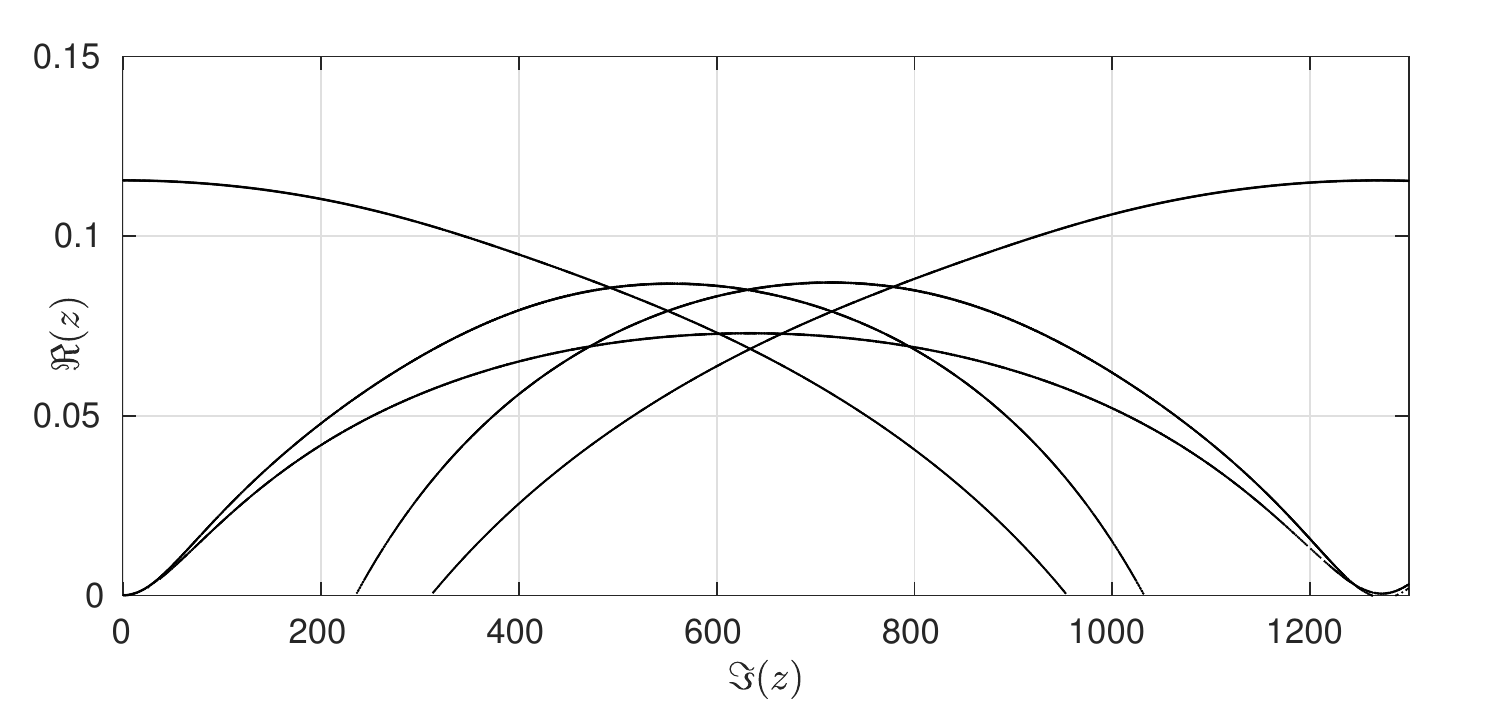} \\
(c) $Z_8(s)$ & (d) $Z_{10}(s)$ \\ 
\includegraphics[scale=0.500,angle=0 ]{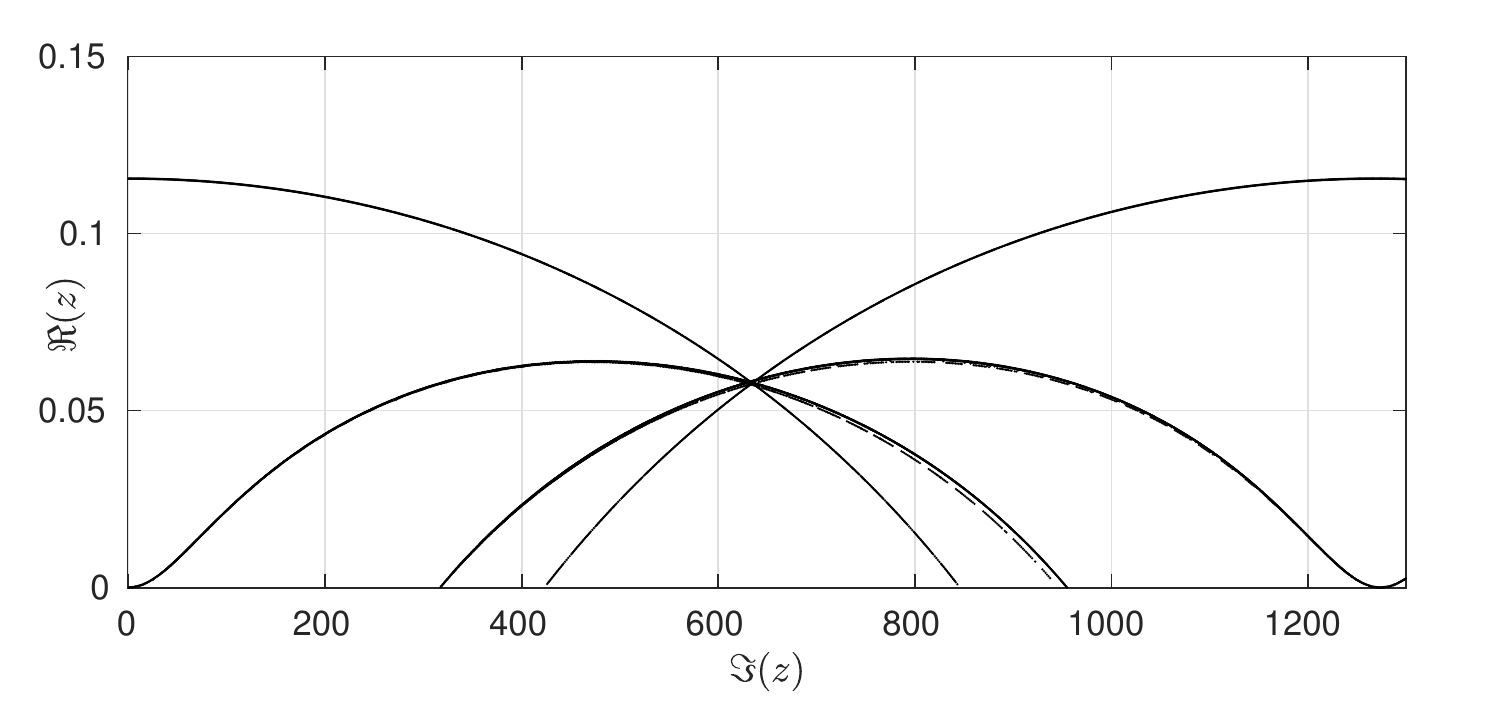} &
\includegraphics[scale=0.500,angle=0 ]{b6n14.pdf} \\
(e) $Z_{12}(s)$ & 
(f) $Z_{14}(s)$ \\ 
\end{tabular}
\end{center}
\caption[Zeros of subsequent approximations]{Plots of the zeros $Z_{n}(s)$ for
$n=2, 4,\ldots, 14$ and $b=6$.}
\label{fig:mult}
\end{figure}
\begin{rem}
  In practice, Theorem~\ref{thm:main-plus} shows that numerical results obtained for
$Z_{n}$ hold  in a domain~$\mathcal R(e^{k_1b})$ for the Selberg zeta function, too. 
Since in practice~$n$ is  bounded above  by  computational
considerations, we may assume that it is fixed. Even with a modern computer,
one will not be able to consider $n>16$ in a reasonable
time\footnote{The most time-consuming part is the Newton method used to locate a
zero starting from a point of a lattice on $\mathcal R(T)$. The time taken by this
calculation grows exponentially with~$n$. The total number of the searches
is proportional to the area of $\mathcal R(T)$, which is proportional to~$T$.}. 
Moreover, in practical
applications~$b$ cannot be chosen too large either due to computer restrictions
because of accumulation of errors while dealing with small numbers. The
coefficients~$b_n$ defined by~\eqref{eq:zcoefs} involve a sum of
$4^{n}+2$ terms of the order $\exp(-2snb)$ with $0<\Re(s)<0.25$, say. 
The bound $4^{n}+2$ is equal to number of closed geodesics of the word
length~$2n$, see Remark~\ref{lem:gcount}. 
\end{rem}





\subsection{Proof of the approximation result}
\label{proofs}

In this section we give a proof of Theorem~\ref{thm:main-plus}.
We will need the following simple technical estimate.
\begin{lemma}
  \label{lem:erf}
  Let $x_n$ be a sequence of real numbers satisfying $|x_n|\le \exp\bigl(p
  n -q n^2\bigr)$ for some constants $p,q>0$. Then for any $n>1$ we have
  that 
  $$
  \sum_{k=n}^\infty |x_k| \le \frac{\sqrt\pi}{2\sqrt{q}}
  \exp\Bigl(\frac{p^2}{4q} \Bigr) \exp\Bigl(-q\Bigl(n -
  \frac{p}{2q} \Bigr)^2\Bigr)
  $$
\end{lemma}
\begin{proof}
  The result follows by straightforward calculation using the classical bound
  for the error function $\int_{n}^\infty \exp(-t^2) d\,t \le \frac{\sqrt\pi}{2}
  \exp(-n^2)$.
\end{proof}
We now turn to the proof of Theorem~\ref{thm:main-plus}. 
This follows the same lines as~\cite{JP02}.
However, the key new idea is that the disks~$U_1$,~$U_2$ and~$U_3$ used to
define~$\mathcal B$ are allowed to depend on~$b$.  

\begin{proof}[of Theorem~\ref{thm:main-plus}]
Without loss of generality we may
assume that the geodesic $\beta_j$ has end points
$e^{(\frac{2\pi}3j\pm\theta)i}\in\partial\bbD^2$ for $j=1,2,3$. We choose three additional geodesics $\upsilon_j$
with end points  $e^{(\frac{2\pi}3j\pm\varphi)i}\in\partial\bbD^2$ for some $0<\varphi<\theta$,
that we will specify later, see Figure~\ref{fig:D} for details. We may
consider $\bbD^2$ as a subset of $\bbC$ with usual Euclidean metric, and then complete
$\upsilon_j$ to full Euclidean circles $\overline\upsilon_j\in\bbC$. We define $U_j\in\bbC$ to be
compact disks with $\partial U_j=\overline{\upsilon}_j$. 

It turns out that the calculations are much easier in the upper half plane model of
the hyperbolic space~$\mathbb H^2$. We choose the map $S(z) = i \frac{1-z}{1+z}$
to change the coordinates. Then the geodesic $\beta_j$ has end points 
$\frac{\sin(\frac{2\pi j}3\pm\theta)}{1+\cos(\frac{2\pi
  j}3\pm\theta)} \in \partial\bbH^2$, and its Euclidean radius and centre are given
  by, respectively 
  \begin{align}
    \varepsilon_j &= \frac12\left(\frac{\sin(\frac{2\pi j}3+\theta)}{1+\cos(\frac{2\pi
  j}3+\theta)} - \frac{\sin(\frac{2\pi j}3-\theta)}{1+\cos(\frac{2\pi
  j}3-\theta)} \right) 
  \label{eq:ej2} \\
    c_j &= \frac12\left(\frac{\sin(\frac{2\pi j}3+\theta)}{1+\cos(\frac{2\pi
    j}3+\theta)} + \frac{\sin(\frac{2\pi j}3-\theta)}{1+\cos(\frac{2\pi
    j}3-\theta)} \right). \label{eq:cj2}
  \end{align}
  The end points of $\upsilon_j$ are $\frac{\sin(\frac{2\pi j}3\pm\varphi)}{1+\cos(\frac{2\pi j}3\pm\varphi)}
  \in \partial\bbH^2$ and the euclidean radius is 
  \begin{equation}
   r_j = \frac12\left(\frac{\sin(\frac{2\pi j}3+\varphi)}{1+\cos(\frac{2\pi
  j}3+\varphi)} - \frac{\sin(\frac{2\pi j}3-\varphi)}{1+\cos(\frac{2\pi
  j}3-\varphi)} \right) 
  \label{eq:rj2}
  \end{equation}
We can consider the Banach space~$\mathcal B$ to be the space of bounded analytic functions on
$\sqcup_{j=1}^3 U_j \subset \bbC$ with the supremum norm. 

We see that the reflection~$R_j$ with respect to the geodesic~$\beta_j$ in~$\bbH^2$
is given by $R_j(z) = \frac{\varepsilon_j^2}{z-c_j}+c_j$. We deduce that for
any distinct $j,k,l$ the image $R_j(U_k \cup U_l) \subset U_j$, provided $
\frac{\varepsilon_j^2}{z-c_j}< r_j$ for any $z \in U_k \sqcup U_l$.  We know that for
all $z \in U_k \sqcup U_l$ we have $|z-c_j|>1$, thus it is sufficient to chose
$\theta$ and $\varepsilon_j$ such that $\varepsilon_j^2 < r_j$. 

Using~\eqref{eq:ej2} by a straightforward calculation we may estimate $
\frac12\theta \le \varepsilon_j \le 2 \theta + O(\theta^2)$ and similarly
from~\eqref{eq:rj2} we have that  $ \frac12
\varphi \le r_j \le 2 \varphi + O(\varphi^2)$ for small 
values of~$\theta$ and $\varphi$. Hence it is sufficient to choose $\theta$ and
$\varphi$ such that $4\theta^2<\frac12 \varphi$. 
Using Lemma~\ref{b-lemma}, we see 
$$
\theta = e^{-b}(1+e^{-2b}+o(e^{-3b})).
$$
In particular, it is sufficient to choose 
\begin{equation}
  \label{phi-def:eq}
  \varphi = e^{-b \varkappa} \mbox{ for some } 1<\varkappa<2 \mbox{ and } b
  \mbox{ sufficiently large}. 
\end{equation}
Then~\eqref{eq:ej2} and~\eqref{eq:rj2} give estimates for the radii of inner and
outer circles, respectively,
\begin{align}
  \frac12 e^{-b\varkappa} &\le r_j \le 2 e^{-b \varkappa} + O(e^{-2 b\varkappa}),
  \label{eq:ej3} \\
  \frac12 e^{-b} &\le \varepsilon_j \le 2 e^{-b} + O(e^{-2b}) .
  \label{eq:rj3}
\end{align}

Using the Cauchy integral formula for $z \in U_k$ and $R_j(z)\in U_j$ we can write 
$$
  f(R_j(z)) = \frac{1}{2\pi i} \int_{\partial U_j } \frac{f(\xi)}{\xi - R_j(z)} d\xi
$$
and thus 
$$
(R_j^\prime(z))^s   f(R_j(z)) = \frac{(R_j^\prime(z))^s}{2\pi i} \int_{\partial
U_j} \frac{f(\xi)}{\xi - R_j(z)} d\xi.$$
Since $\mathcal L_s$ is a nuclear operator, it the identity~\eqref{eq:nucOp}
should hold. 
More precisely, we may write 
\begin{align*}
  (  \mathcal L_s f)(z) &= \sum_{j=1}^3 
  (R_j^\prime(z))^s   f(R_j(z)) \sum_{k=1, k\ne j}^3\chi_{U_k}(z) \\&= 
  \sum_{j=1}^3  (R_j^\prime(z))^s
 \cdot \Bigl(\sum_{n=0}^\infty  \frac{(R_j(z)-c_j)^n}{2 \pi i} \int_{\partial U_j} 
  \frac{f(\xi)}{(\xi-c_j)^{n+1} }  d \xi \Bigr) \cdot \Bigl(
\sum_{k=1, k \ne j}^3   \chi_{U_k}(z) \Bigr) \\ &
= \sum\limits_{n=0}^\infty \lambda_n  w_n(z) \nu_n (f), 
\end{align*}
where, $w_n \in \mathcal B$, $\nu_n \in \mathcal B^*$, and $\lambda_n \in \bbR^+$ satisfy conditions of
Definition~\ref{def:nucOp}. We may choose for any $j \in\{1,2,3\}$ 
\begin{align}
  w_{3n+j}(z)&\asymp(R_j^\prime(z))^s  \cdot  \frac{(R_j(z)-c_j)^n}{2 \pi i} 
\sum_{k=1, k \ne j}^3   \chi_{U_k}(z) \label{eq:wdef2} \\ 
  \nu_{3n+j}(f)&\asymp\int_{\partial U_j} 
  \frac{f(\xi)}{(\xi-c_j)^{n+1} }  d \xi, 
  \label{eq:nudef2}
\end{align}
with normalization $\|w_n\|_\infty = \|\nu_n\|_\infty=1$, and where $c_j$ are
given by~\eqref{eq:cj2}. Then for any $j \in
\{1,2,3\}$ 
\begin{equation}
  \label{eq:lamDef2}
  |\lambda_{3n+j}| = \left\|(R_j^\prime)^s\mid_{U_k\cup U_l} \|_\infty \cdot
  \|(R_j-c_j)^n\mid_{U_k\cup U_l} \right\|_\infty \cdot 
  \left\|\frac{1}{2\pi i}  \int_{\partial U_j} 
  \frac{f(\xi)}{(\xi-c_j)^{n+1}}d\xi \right\|_\infty,
\end{equation}
where $k\ne j$ and $l \ne j$, and the latter term is the supremum norm of the
functional
$$
f \mapsto \frac{1}{2\pi i}  \int_{\partial U_j} 
  \frac{f(\xi)}{(\xi-c_j)^{n+1}}d\xi.  
$$
We may observe that for any $\xi \in \partial U_j$ one has that $\xi-c_j=r_j$ and conclude
$$
|\lambda_{3n+j}| \le \|R_j^\prime(z)^s\mid_{U_k\cup U_l} \|_\infty \cdot
\frac{\varepsilon_j^{2n}}{r_j^{n+1}}.
$$
More precisely, using formulae~\eqref{eq:ej2} and~\eqref{eq:rj2}, we obtain
an upper bound
\begin{align*}
  |\lambda_n| &\le \max_j \|R_j^\prime(z)^s\mid_{U_k\cup U_l} \| \cdot\max\Bigl\{ \frac{1}{r_1}
  \frac{\varepsilon_1^{2n/3}}{r_1^{n/3}}, \, \frac{1}{(\varepsilon_2 r_2)^{2/3}} 
  \frac{\varepsilon_2^{2n/3}}{r_2^{n/3}}, \, \frac{1}{\varepsilon_3^{4/3} r_3^{1/3}}
  \frac{\varepsilon_3^{2n/3}}{r_3^{n/3}}    \Bigr\}  \\&\le
  \max_j \|R_j^\prime(z)^s\mid_{U_k\cup U_l} \| \cdot
  \frac{1}{\varepsilon_3^{4/3} r_3^{1/3}} \max_j
  \frac{\varepsilon_j^{2n/3}}{r_j^{n/3}}.  
\end{align*}
Comparing this with the definition of the nuclear operator~\ref{def:nucOp}, we get
explicit bounds for parameters~$\lambda$ and $C(s)$.
$$
  |\lambda_n| 
  \leq   \max_k \sup_{z \in U_k} \max_{j\ne k}\left|
  \frac{\epsilon_j^{2s}}{(z-c_j)^{2s}}\right| \cdot
  \frac{1}{\varepsilon_3^{4/3} r_3^{1/3}} \cdot \max_j
  \frac{\epsilon_j^{2n/3}}{r_j^{n/3}} \le C(s) \lambda^n,
$$
  with the choices 
\begin{align}
  \lambda &= \Bigl( \max_{j}\frac{\epsilon_j^2}{r_j}\Bigr)^{\frac13}  \label{eq:lam2} \\ 
 C(s) & =  \max_{k} \sup_{z \in U_k} \max_{j \ne k} \left|
 \frac{\epsilon_j^{2s}}{(z-c_j)^{2s}} \right|\cdot 
 \frac{1}{\varepsilon_3^{4/3} r_3^{1/3}}. \label{eq:Cs2}
\end{align}
Using the bounds~\eqref{eq:ej3} and~\eqref{eq:rj3} for $\varepsilon_j$ and
$r_j$, we conclude 
\begin{equation}
  \label{eq:lup}
  \lambda=\Bigl(\max\frac{\varepsilon_j^2}{r_j}\Bigr)^{\frac13} \le 2
  e^{-\frac{2-\varkappa}{3}b}.
\end{equation}
Furthermore, we see that for any $j \ne k$ for all $z \in U_k$ we have
$|\arg(z-c_j)|\le \arcsin (\frac{r_j}{c_k-c_j}) \le 2\varphi$. Therefore for $s=\sigma+it$,  
\begin{align*}
|(z-c_j)^{2s}| &=\bigl|\exp\bigl(2 (\ln|z-c_j|+i\arg(z-c_j))\cdot (\sigma+it)
\bigr)\bigr| \\ &= |z-c_j|^\sigma\cdot\exp(-2\arg(z-c_j)t) \ge \exp(4\varphi t),
\end{align*}
since by construction $\inf_{z \in U_k}|z-c_j| > 1$. Using~\eqref{eq:ej3}
and~\eqref{eq:rj3}, we deduce
\begin{multline}
  \label{eq:Csup2}
C(s) = \max_{k} \sup_{z \in U_k} \max_{j \ne k} \left| \frac{\epsilon_j^{2s}}{(z-c_j)^{2s}}\right| \cdot
\frac{1}{\varepsilon_3^{4/3} r_3^{1/3}} \le  \max_{k} \sup_{z \in U_k} \max_{j \ne k} 
\left| \frac{\epsilon_j^{2\sigma}}{(z-c_j)^{2s}}\right| \cdot 4
e^{b(\varkappa+4)/3}  \\ 
\le 4 \epsilon_j^{2\sigma} e^{4 \varphi t + b(\varkappa+4)/3} \le e^{\ln 4 -2b\sigma +
4\varphi t + b(\varkappa+4)/3}.  
\end{multline}
Substituting bounds~\eqref{eq:lup} and~\eqref{eq:Csup2} into Ruelle's
inequality~\eqref{eq:anb} and taking into account $t<T$ for $s = \sigma+it \in \mathcal
R(T)$, we obtain an upper bound
\begin{align}
  |a_n(s)| &\le C^n(s) \lambda^{n(n+1)/2} n^{n/2} \notag \\ &\le \exp\Bigl( (\ln 4 - 2b\sigma +
  \frac{b(\varkappa+4)}{3} +
  4\varphi t)n - \frac{n(n+1)}{2}\Bigl( \frac{b (2-\varkappa)}3 - \ln 2 \Bigr) + \frac{n \ln n}{2}
  \Bigr)  \notag \\ 
  &\le \exp\Bigl( (\ln 4 + \frac{b(\varkappa+4)}{3} + 4\varphi T) n - \frac{n(n+1)}{2}\Bigl(\frac{ b
  (2-\varkappa)}{3}- \ln 2  \Bigr) + \frac{n \ln n}{2}
  \Bigr), \label{eq:an2}
\end{align}
since $\exp(-2bn\sigma)\le 1$.

In order to estimate the tail of the series $\sum\limits_{n=14}^\infty a_n(s)$ using
Lemma~\ref{lem:erf}, it is sufficient to find a constant
$k_2<1$ such that 
\begin{equation}
  \label{eq:k5kap}
n\Bigl(\ln 4 + \frac{b(\varkappa+4)}{3} + \frac{ b (2-\varkappa)}{6}\Bigr)+ \frac{(n+1)n}{2} \ln 2 + \frac{n \ln n}{2} <
\frac{ b n^2  (2-\varkappa)}{6} k_2
\end{equation}
which is equivalent 
\begin{equation}
  \frac{\frac{15\ln 2}{b}+10+\varkappa}{k_2(2-\varkappa)-\frac{3\ln 2}{b}} < n -
  \frac{3 \ln n}{b k_2 (2-\varkappa) - 3 \ln 2}. 
\end{equation}
It is clear that the last inequality doesn't hold for any $n \le 10$ but it does
hold, for example,
for all $b\ge 20$ and $n\ge14$  with the choices $\varkappa=1.05$ and
$0.95\le k_2 < 1$.
Therefore we obtain an upper bound 
\begin{equation}
  |a_n(s)| \le \exp\Bigl(4 \varphi T n - \frac{ b  (2-\varkappa)(1-k_2) }{6} n^2
  \Bigr), \quad \mbox{ for all } n\ge 14.
\end{equation}
We recall that $\varphi=e^{-\varkappa b}$ by~\eqref{phi-def:eq} and applying Lemma~\ref{lem:erf} with
the choices $p=4 e^{-\varkappa b} T$ and $q=\frac{b(2-\varkappa)(1-k_2)}{6} $,
we get an estimate 
$
\sum\limits_{k=n}^\infty |a_n(\sigma+it)| \le \eta(b,n,T(b))$,
where $T(b) = k_0 e^{\varkappa b}$ for some $k_0>0$, all $b\ge20$, $n\ge14$ and
\begin{multline}
\eta(b,n,T(b)) = \\ = \frac{\sqrt{6\pi}}{2\sqrt{b(2-\varkappa)(1-k_2)}}
\exp\Bigl(\frac{24 k_0^2 }{b(2-\varkappa)(1-k_2)}\Bigr)
\exp\Bigl(\frac{b(2-\varkappa)(k_2-1)}{6} \Bigl(n-\frac{12
k_0}{b(2-\varkappa)(1-k_2)}\Bigr)^2\Bigr). 
\label{eq:remb}
\end{multline}
Therefore we have the desired asymptotic estimates:
\begin{enumerate}
  \item for any  $n \ge 14$ we have $ \eta(b,n,T(b)) = O\bigl( \frac{1}{\sqrt b}
    \bigr) $ as $ b \to \infty$;
  \item for any  $b \ge 20$ we have  $ \eta(b,n,T(b)) = O\bigl( e^{-bk_1n^2}
    \bigr)$;
   as $ n \to \infty$.
\end{enumerate}
hold with the choices, for example, 
$0<k_1 \le \frac{(2-\varkappa)(1-k_2)}{6}$, where 
$1<\varkappa<2$ and $k_2$ are chosen so that~\eqref{eq:k5kap} holds.

\end{proof}


\section{Results on the zero set}

We now turn to the problem of describing the distribution of the zeros.
In Section~\ref{ss:exp} we introduced a matrix function~$A(s)$, closely connected to
the zeta function. In the following proposition we study the convergence of
$$
\lim_{n \to \infty}
     \det\left(I_N - A^2(s)\right);
$$
from Lemma~\ref{PPT:lem}.

Let us recall the matrix $A(s)$ computed in~\eqref{A2-def:eq} using an
approximation to the length of closed geodesics based on the segments of word
length~$2$: 
$$
A^2(s) = 
e^{-2sb} \left(\begin{matrix}
  1 & z & 0 & 0 & z^2 & z \cr 
  z & 1 & z^2 & z & 0 & 0 \cr
  0 & 0 & 1& z& z & z^2 \cr
  z^2 & z & z& 1& 0 & 0 \cr
  0 & 0 & z & z^2 & 1 & z \cr
  z& z^2& 0 & 0 & z & 1 \cr
\end{matrix}\right), \qquad \mbox{ where } z=e^{-se^{-b}}.
$$
As we are looking to study rescaled zeros, 
 $$  
 \left\{\sigma b+i e^{-b} t \Bigl| Z_{X_b} (\sigma+it) = 0 \right\} =
 \left\{\sigma + it \Bigl| Z_{X_b} \left( \frac{\sigma}{b} + it e^b \right) = 0
 \right\},
 $$
it is appropriate to consider 
$A\left(\frac{\sigma}{b} + it e^b \right) = e^{-2\sigma - 2it b e^b} B(z)$, where
$z= \exp\left(-\frac12\left(\frac{\sigma}{b e^b} + it\right)\right)$ and   $B(z)
= e^{2bs} A^2(s)$, defined by~\eqref{B-mat:eq}. 
Taking into account that $\exp\left(-\frac{\sigma}{2 b e^b}\right) \to 1$
as $b \to + \infty$ for $\sigma>0$, we conclude the following:

\begin{prop}
  \label{thm:app2}
  Using the notation introduced above, the real analytic function $Z_{12}\bigl(\frac{\sigma}{b} +
ite^b\bigr)$ converges uniformly to $\det(I-\exp(-2\sigma-2itbe^b)B(e^{it}))$ on
the critical strip, and more precisely, 
$$
\left|Z_{12}\left(\frac{\sigma}b + i
t e^{b}\right) -\det\left(I-\exp(-2\sigma-2itbe^b)B(e^{it})\right)  \right|  =
O\left(e^{-b}\right)
\mbox{ as }  b \to +\infty.  
$$
\end{prop}
\begin{proof}
  This follows by straightforward calculation of the first $12$ coefficients and
  the determinant. Let us introduce dummy variables
  $x=\exp(-2\sigma-2itbe^b)$ and $y=e^{it}$ with $|x|< 1$ and
  $|y|=1$. Then 
  \begin{equation}
    \det\left(I-\exp(-2\sigma-2itbe^b)B(e^{it})\right) = \det\left(I-xB(y)\right)
= \sum_{k=0}^6 x^k P_k(y^2),
  \end{equation}
  where $P_k \in \bbZ[\cdot]$ are some polynomials with integer coefficients. 
  More precisely, we can compute:
  \begin{align*}
    P_0(y) &\equiv 1, \\
    P_1(y) & \equiv 6, \\
    P_2(y) & = 15 - 6y^2, \\
    P_3(y) & =  20 - 24y^2 + 6y^4 + 2y^6,  \\
    P_4(y) & = 15 - 36y^2 + 27y^4 - 6y^6, \\
    P_5(y) & = -6(y^2-1)^4, \\
    P_6(y) & = (y^2-1)^6. 
\end{align*}
 On the other hand
 $$
 Z_{12}\left(\frac{\sigma}b + i t e^{b}\right) = \sum_{j=0}^{12} a_j\left(\frac{\sigma}b + i t
 e^{b}\right)
 $$
Now one can deduce by comparing coefficients in $x^n$ that
$a_{2n}\left(\frac{\sigma}b + i t e^{b}\right) = x^n P_n(y) + O(e^{-b})$. 
\end{proof}

\begin{rem}
  \label{rem:coefs}
Using the estimates for the hyperbolic length of closed geodesics of the word length
$\omega(\gamma)\le6$, presented in the Appendix~\ref{example-ap:s}, we may
explicitly compute the first few
non-zero coefficients  
\begin{align*}
a_2(s)& = - \frac{6 e^{-2bs}}{1-e^{-2b}} \\
a_4(s)& = 15e^{-4bs} - 6 e^{-4bs}  \bigl(1+e^{-b} + 2e^{-2b} \bigr)^{2s} +
O(e^{-b}) \\
a_6(s)& = - 20e^{-6bs} - 2 e^{-6bs} \bigl(1+3e^{-b}+3e^{-2b}\bigr)^{2s} -
  6 e^{-6bs}  \bigl(1 + 2 e^{-b} + 3e^{-2b} \bigr)^{2s} + \\ & +
   24e^{-6bs}  \bigl(1+e^{-b}+ 2e^{-2b}\bigr)^{2s} +  O(e^{-b}). 
 \end{align*}
We shall illustrate Proposition~\ref{thm:app2} using formulae for the
coefficients above. We can write 
\begin{align*}
a_2\left(\frac{\sigma}b + i t  e^{b}\right) &= 
-\frac{6\exp\bigl(-2\sigma-2itbe^b\bigr)}{1-e^{-2b}}=\frac{-6x}{1-e^{-2b}} = -6x
(1+O(e^{-2b}))= xP_1+O(e^{-b}). \\ 
\intertext{ Similarly for $a_4$:}
a_4\left(\frac{\sigma}b + i t  e^{b}\right) &= 
15e^{-4\sigma-4itbe^b} - 6 e^{-4\sigma-4itbe^b}  \bigl(1+e^{-b} + 2e^{-2b}
\bigr)^{2\sigma/b+2ite^b} + O(e^{-b})  \\ &= 15x^2 - 6x^2\bigl(1+e^{-b} + 2e^{-2b}
\bigr)^{2\sigma/b}\cdot \bigl(1+e^{-b} + 2e^{-2b}
\bigr)^{2ite^b} + O(e^{-b})  \\ &= 
15x^2 - 6x^2e^{2it}\cdot\bigl(1+O(e^{-b})\bigr)+
O(e^{-b})  \\ & = x^2 \bigl(15-6y^2\bigr) + O(e^{-b})= x^2 P_2(y)+O(e^{-b}),
\end{align*}
where we have used the fact that $\bigl(1+e^{-b} + 2e^{-2b}
\bigr)^{e^b} = e + O\left(e^{-b}\right) $.
\end{rem}

Now we can prove Theorem~\ref{thm:bor}.
\begin{proof}[of Theorem~\ref{thm:bor}] 
We find that the matrix $B(e^{it})$ has exactly
four different eigenvalues~$\mu_k(t)$, $k=1,\ldots,4$: 
\begin{align*}
\mu_1(t)&=(e^{it} - 1)^2 \\
\mu_2(t)&= (e^{it} + 1)^2\\
\mu_3(t)&=1 - \frac{e^{2it}}{2} + e^{it}\frac{\sqrt{4 - 3e^{2it}}}{2} \\
\mu_4(t)&= 1 - \frac{e^{2it}}{2} - e^{it}\frac{\sqrt{4 - 3e^{2it}}}{2}.
\end{align*}
By Lemma~\ref{det0:lem}, the zero set of the determinant 
$\det\left(I-\exp(-2\sigma-2itbe^b)B(e^{it})\right)$  belongs to the subset 
$\left\{ (\sigma,t)\in\bbR^2 \mid \exists k \colon |\exp(2\sigma + 2itbe^b)| =
|\mu_k| \right\}$.
The four equations $\exp(2\sigma)= |\mu_k(t)|$ give us four curves
$$
\begin{aligned}
\mathcal C_1 &= \left\{ 
\frac12\ln|2-2\cos(t)| + it\mid t \in \mathbb R \right\}; \cr 
\mathcal C_2 &= 
\left\{ \frac12 \ln|2+2\cos(t)| + it \mid t \in \mathbb R \right\};  \cr  
\mathcal C_3 &= \left\{ \frac12 \ln \left|
1 - \frac12 e^{2it} - \frac12 e^{it} \sqrt{4 - 3 e^{2i t}} \right| +it \mid t
\in \mathbb R \right\}; \cr
\mathcal C_4 &= \left\{  \frac12 \ln \left|
1 - \frac12 e^{2it} + \frac12 e^{it} \sqrt{4 - 3 e^{2i t}} \right| +it  \mid
t \in \mathbb R \right\}.
\end{aligned}
$$

Since the curves $\mathcal C_j$ do not have horizontal tangencies
$\sigma\equiv\mathrm{const}$, without loss of generality we may define 
neighbourhoods as follows:
$$
V(\mathcal C_j,\varepsilon)\mbox{:} =  \left\{(\sigma,t) \mid
\left|2\sigma-\ln|\mu_j(t)|\right|<2\varepsilon \right\}.
$$
To complete the argument we shall show that
for all $\varepsilon>0$ and $T>0$ there exists
$b_0 > 0$ such that for any $b > b_0$ the zeros of the function
$Z\left(\frac{\sigma}{b} + it e^{b}\right)$ with $0 \leq \sigma \leq 1$ and $|t|
\leq e^{(2-\varkappa)b}$ belong to a neighbourhood $\cup_k V(\mathcal C_k,
\varepsilon)$ of the union of the curves $\cup_k \mathcal C_k$.

 Indeed,  given $\varepsilon>0$ and a point $z_0=\sigma_0+it_0$ outside of $\varepsilon$-neighbourhood
  of $\cup_{j=1}^4 \mathcal C_j$ 
  we see that the determinant
  $$
  \left|\det(I-\exp(-2\sigma_0-it_0be^b)B(\exp(it_0)))\right|>\exp(-6\varepsilon)(\exp\varepsilon-1)^6>0
  $$ 
  is bounded away from zero and the
  bound is independent of $b$. Summing up, we see that outside of the neighbourhood $\cup_{j=1}^4 V(\mathcal
  C_j,\varepsilon)$ the determinant has modulus uniformly bounded away from~$0$;
  by Theorem~\ref{thm:main-plus} for~$b$ large we have that the zeta function
  $Z_{X_b}\left(\frac{\sigma}b + it e^b\right)$
  can be approximated by $Z_{12}$ arbitrarily closely and by
  Proposition~\ref{thm:app2} $Z_{12}$ can be approximated arbitrarily closely by the determinant.  
  Therefore for $b$ sufficiently large all zeros of the function
  $Z_{X_b}\left(\frac{\sigma}{b} + it e^b\right)$ belong to the $\varepsilon$-neighbourhood of
  $\cup_{j=1}^4 \mathcal C_j$.
  \end{proof}

We have concentrated on the particular case of the symmetric $3$-funnelled
surface (whose defining closed geodesics have the same lengths).  However, the same method of
combining geometric and analytic approximations  works
in the case that the boundary curves have different length as well as in the
case of symmetric punctured torus, and allows one to explain the nature of the 
patterns of zeros described in the sections 5.1 and 5.2 of~\cite{B14}.

\section{$L$-functions and  covering surfaces}
Our results have concentrated on a special class of surfaces, but can be easily
adapted to cover a large class of geometrically finite surfaces of infinite area. 

There is a fairly simple method for constructing quite complicated surfaces
using any (infinite area) surface~$V$ without cusps. We can write $V = \mathbb
D^2/\Gamma$ for a convex cocompact group~$\Gamma$.  Then 
we can define a (finite) cover $\widehat V = \mathbb H^2/\Gamma_0$ for~$V$ in
terms of a (finite index) normal subgroup $\Gamma_0 < \Gamma$. 


Let us denote by $G = \Gamma/\Gamma_0$ the finite quotient group.    
Let $\gamma$ be a closed geodesic on $X_b$  and then this is covered by the union
of closed geodesics  $\gamma_1, \cdots, \gamma_n$ on $\widehat V$.

Let  $R_\chi$ be an irreducible representation for $G$  of degree $d_\chi$
with character $\chi = \hbox{tr}(R_\chi)$. The regular representation of $G$
can be written $R = \oplus_\chi d_\chi R_\chi.$ where $|G| = \sum\limits_\chi d_\chi^2$

\begin{definition}
Given $s\in \mathbb C$ we define 
$$
L(z, s, \chi) = \prod_{\gamma} \det \left(I - z^{|g|}e^{-(s+n)\lambda(g)} R(g\Gamma_0)\right)
$$
where $g\Gamma_0$ is a coset in~$G$.
\end{definition}

\begin{lemma}
For characters $\chi_1$ and $\chi_2$ we can write 
$$
  L(z, s,\chi_1 + \chi_2) = L(z, s,\chi_1) L(z, s,\chi_2).
$$
\end{lemma}
  
If $H < G$ is a subgroup and~$\chi$ is a character of~$H$ then we can write 
$G = \cup_{i=1}^m H\alpha_i$ and define the induced character~$\chi^*$ of~$G$ by 
$$
\chi^*(g) = \sum\limits_{\alpha_i g \alpha_i^{-1} \in H} \chi(\alpha_i g \alpha_i^{-1})
$$
for $g\in G$.
 
\begin{lemma}[Brauer--Frobenius]
 Each non-trivial character $\chi$ is a rational combination of characters $\chi_i^*$ of $G$ 
 induced from non-trivial characters $\chi_i$ of cyclic subgroups $H_i$.
\end{lemma}
There exist integers $n_1, \cdots, n_k$  with 
$$
  n\chi = \sum\limits_{i=1}^k n_i \chi_i^*
$$
and thus 
$$
L(s, z, \chi)^n = \sum\limits_{\alpha_i g \alpha_i^{-1} \in H}  \chi(\alpha_i g
\alpha_i^{-1}).
$$
Since it is easier to deal with cyclic covering groups.  We need the following.

\begin{lemma}
Let $\chi$ be a character of the subgroup $H < G$ and let $\widehat L(s,z,
\chi^*)$ be the
$L$-function with respect to the covering  $\widehat V$ of  $\widehat V/H$.
Then $L(s, z, \chi) = \widehat L(s, z, \chi^*)$. 
\end{lemma}

The proof is analogous to that of the proof of Proposition 2 in~\cite{PP86}.

This leads to the following.

\begin{lemma}
If $\chi$ is an irreducible non-trivial character of $G$ then $L(s, \chi)^n$ is
a product of integer powers of $L$-functions defined with respect to non-trivial
characters of cyclic subgroups of $G$. 
\end{lemma}

Finally this means that we can write the zeta function $Z_{\widehat V}(s,z)$
in terms of the $L$-functions $L(s,z, \chi)$ for $V$.

\begin{lemma}
We can write
$$
Z_{\widehat V}(s,z) = \prod_{\chi \mbox{ irreducible}}  L(s,z, \chi)^{d_\chi}.
$$
where the product is over all irreducible representations of $G$.
\end{lemma}

In particular, the zeros for 
$Z_{\widehat V}(s,z)$ will be the union of the zeros for the  $L$-functions
$L_{V}(s,z, \chi)^{d_\chi}$. 

\begin{example}
We can take a double cover $\widehat X_b$ for a three funnelled surface $X_b$, which
corresponds to a $4$-funnelled surface. The corresponding covering group is
simply  $\mathbb Z_2$ and the  zeta function $Z_{\widehat X_b}(s)$ is then the
product of: 
\begin{enumerate}
\item the zeta function $Z_{X_b}(s)$ for the original surface; 
\item  the $L$-function  $L(s,\chi):= L(s,1,\chi)$ corresponding to the representation
  $\chi: \pi_1(X_b) \to \mathbb Z_2$ where $\chi(g) = (-1)^{n(g)}$ where $n(g)$
  counts the number of times the generator $a$, say, occurs in $g$.
\end{enumerate} 

In particular the zeros for $Z_{\widehat X_b}(s)$ are a union of the figures for
these two functions.  

\begin{figure}
  \begin{tabular}{ccc}
\includegraphics[scale=0.600,angle=0 ]{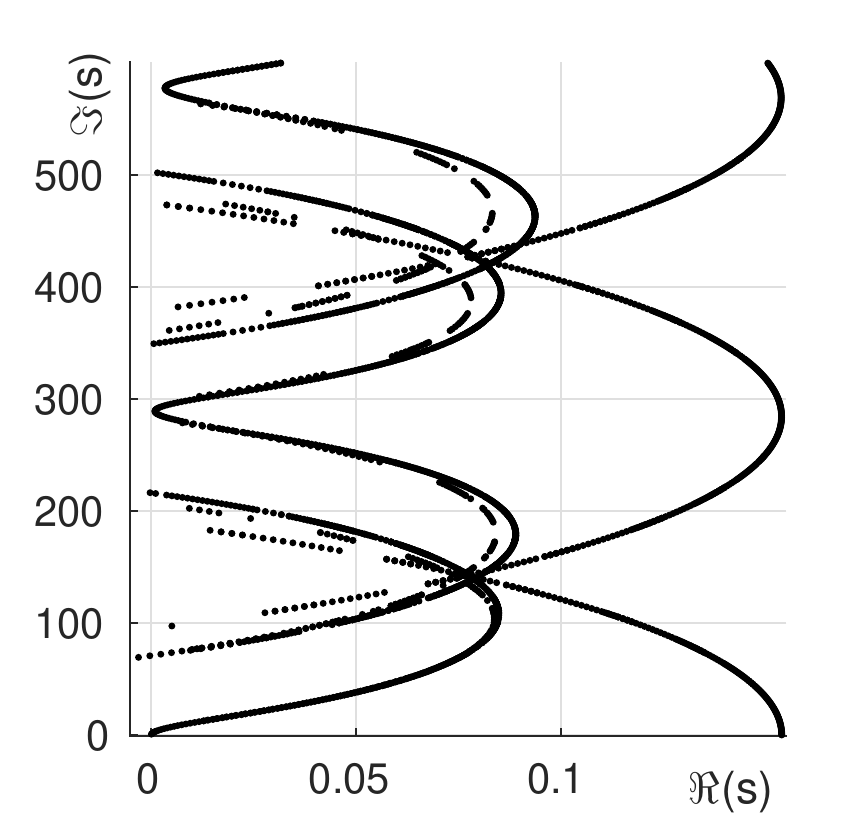} &
\includegraphics[scale=0.600,angle=0 ]{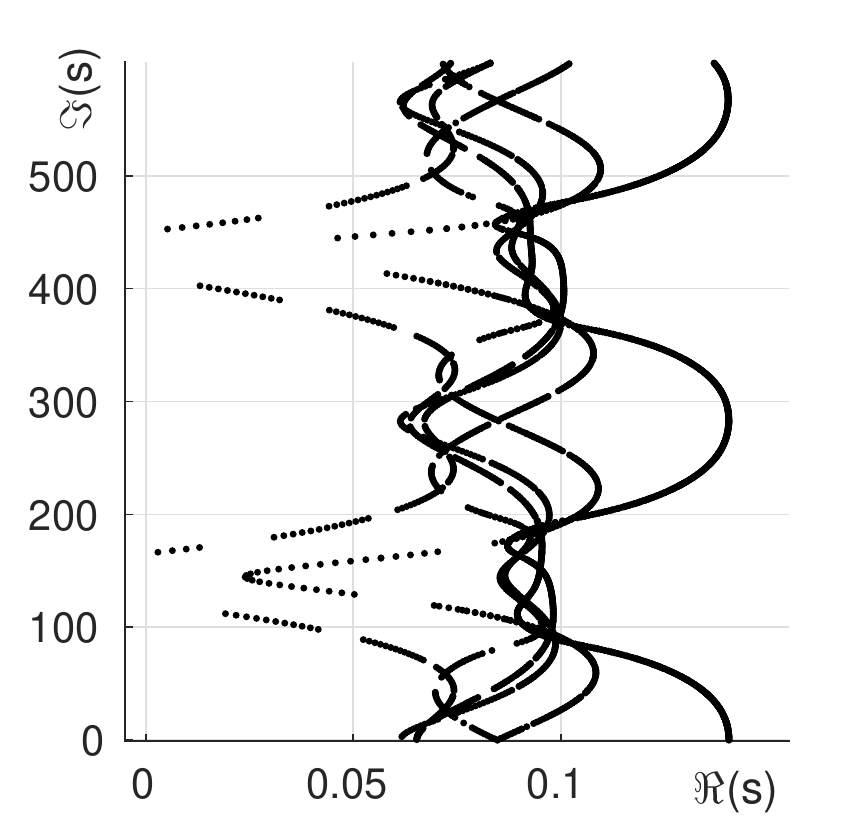} &
\includegraphics[scale=0.600,angle=0 ]{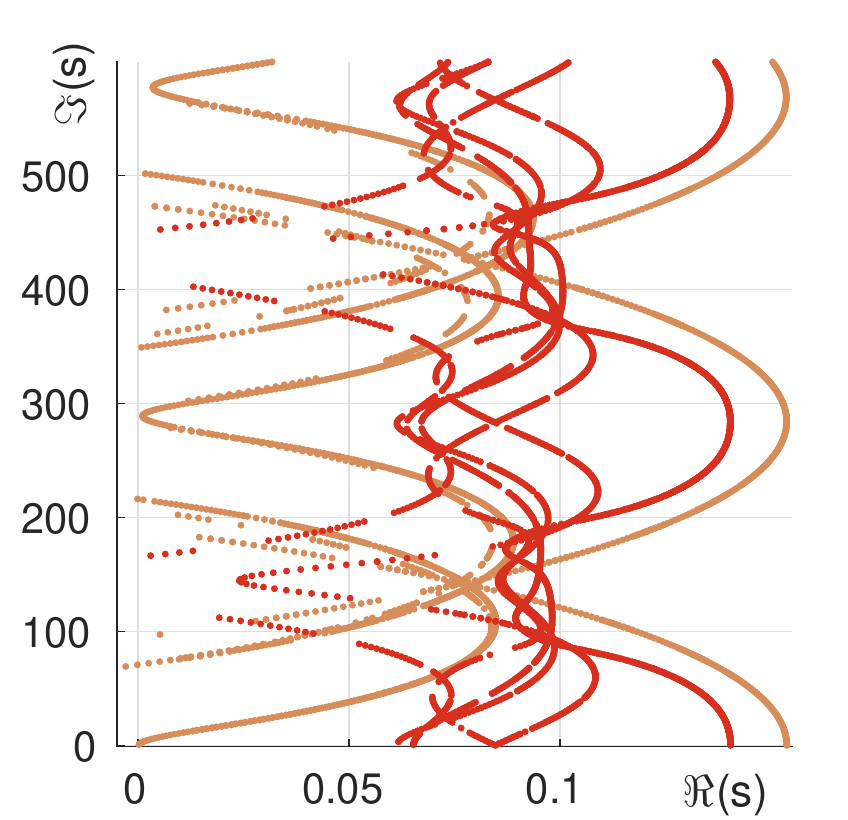} \\
(a) zeros for $Z_{X_b}(s)$ & (b) zeros for
$L_{X_b}(s,z,\chi)$ & (c) superposition. \\
  \end{tabular}
\caption[Pattern of zeros of $L$-functions]{ Zeros of the zeta function and zeros of the $L$-function in the case
$\ell(\gamma_{1,2,3})=9$. The apparent gaps are due to instability
of the Newton method. } 
\end{figure}
\end{example}

\appendix

\section{Examples of the coefficients}
\label{example-ap:s}
In this Appendix we present the asymptotic formulae for the hyperbolic length of
the short closed geodesics, which then lead to the asymptotic expressions for
the first few non-zero coefficients $a_2$, $a_4$, $a_6$.

Using the identity 
$$
\ell(\gamma_{j_1,j_2,\ldots j_{2n}}) = 2 \acosh \left(\frac12 \tr (R_{j_1} R_{j_2}
\ldots R_{j_{2n}} ) \right),
$$
relating the length of the closed geodesic corresponding to
the cutting sequence of period $2n$ to the matrices defining the reflections, we
compute the lengths of the closed geodesics for $n=2,4,6,8$. 

\paragraph{The case $n=2$.} It has been established in Lemma~\ref{lem:lgam}, see
also remark~\ref{lem:gcount}, that there are exactly $6$~geodesics of length
$\ell(\gamma_{j_1 j_2}) = 2b$.

\paragraph{The case $n=4$.} There are $6$~geodesics of  length $4b$ and
$12$~geodesics of length 
$$
\ell(\gamma_{j_1 j_2 j_3 j_2}) = 2\acosh\bigl(\cosh(b)+2\cosh^2(b)\bigr).  
$$

\paragraph{The case $n=6$.} There are~$4^3+2=66$ homotopy classes of closed geodesics;
among which there are $6$~geodesics of length~$6b$ and of length
\begin{align*}
\ell(\gamma_{j_1 j_2 j_3 j_1 j_2 j_3})  &= 
2\acosh\bigl( 4\cosh^3(b) + 6\cosh^2(b) -1\bigr) = 6b + 6 e^{-b} + O(e^{-2b}).\\
\intertext{There are $18$~geodesics of length} 
\ell(\gamma_{j_1 j_2 j_1 j_3 j_2 j_3})  &= 
2\acosh\Bigl(8\cosh^2\Bigl( \frac{b}{2} \Bigr)\cdot\cosh^2(b)-1\Bigr) = 6b + 4
e^{-b} + O(e^{-2b}).\\
\intertext{Finally, there are $36$~geodesics of length} 
\ell(\gamma_{j_3 j_2 j_1 j_2 j_1 j_2})  &= 
2\acosh\bigl( 4\cosh^3(b)+ 2\cosh^2(b)-\cosh(b)\bigr) = 6b + 2e^{-b} +
O(e^{-2b}).
\end{align*}

\paragraph{The case $n=8$.} There are~$4^4+2=258$ homotopy classes of closed geodesics;
among which there are $6$~geodesics of length~$8b$. 
Moreover, there are $24$~geodesics of length
\begin{align*}
\ell(\gamma_{j_3 j_1 j_3 j_1 j_2 j_1 j_2 j_1 }) & =
  2\acosh(-4\cosh^2(b)+4\cosh^3(b)+8\cosh^4(b)+1) = 8b + 2 e^{-b} + O(e^{-2b});
  \\
  \intertext{and another $48$~geodesics of length} 
\ell(\gamma_{j_3 j_1 j_2 j_1 j_2 j_1 j_2 j_1} ) & = 
  2\acosh(-\cosh(b)-4\cosh^2(b)+4\cosh^3(b)+8\cosh^4(b)) = 8b + 2 e^{-b} +
  O(e^{-2b}).\\ 
  \intertext{In addition, we have $12$ geodesics of length}
\ell(\gamma_{j_3 j_1 j_2 j_1 j_3 j_1 j_2 j_1 })  &=
2\acosh(2\cosh^2(b)+8\cosh^3(b)+8\cosh^4(b)-1) = 8b + 4 e^{-b} + O(e^{-2b}); \\
\intertext{and $48$ geodesics of length} 
\ell(\gamma_{j_3 j_1 j_2 j_3 j_1 j_2 j_1 j_2} ) &=  
  2\acosh(\cosh(b)\cdot(4\cosh(2b)+2\cosh(3b)+4\cosh(b)+1)) =  8b + 4 e^{-b} +
  O(e^{-2b}); \\
\intertext{and another $48$ geodesics of length}
\ell(\gamma_{j_3 j_2 j_1 j_3 j_1 j_2 j_1 j_2} ) &=
  2\acosh(-\cosh(b)+8\cosh^3(b)+8\cosh^4(b)) = 8b + 4 e^{-b} + O(e^{-2b}). \\
\intertext{Finally, there are $24$ geodesics of length} 
\ell(\gamma_{j_3 j_2 j_1 j_2 j_3 j_1 j_2 j_1} ) &=
  2\acosh(\cosh(b)\cdot(6\cosh(2b)+2\cosh(3b)+8\cosh(b)+3)) = 8b + 6 e^{-b} +
  O(e^{-2b}); \\
\intertext{and further more $48$ geodesics of length}
\ell(\gamma_{j_3 j_1 j_2 j_3 j_1 j_2 j_1 j_2} )  &=
2\acosh(-3\cosh(b)+12\cosh^3(b)+8\cosh^4(b)) = 8b + 6e^{-b} + O(e^{-2b}).
\end{align*}

\section{Non-periodicity}
The apparent almost periodicity in the plot can never be exact for a fixed $b$ as we see from the
behaviour  of the zeros near the  vertical line~$\Re(s)=\delta$ in
Figure~\ref{fig:zooming}.  In fact, since the geodesic flow restricted to the
non-wandering set is mixing, it is shown in~\cite{PP90} that there is only one zero with
$\Re(s)=\delta$. Naud~\cite{N05} (see also Jacobson \& Naud~\cite{JN16}) showed an even stronger result: there exists 
$\varepsilon > 0$ such that there is only finite number of zeros satisfying $\Re(s)>\delta - \varepsilon$.
This is illustrated by the numerical results in Table~\ref{table:pearls}. Namely, we analyze values of zeros closest to the
right boundary~$\Re(z)=\delta$ of the critical strip:
$$
\mathcal E:=\left\{ s_0 \in \mathcal S_{X_b} \mid \mbox{ for all } s \in \mathcal S_{X_b} \mbox{
such that } | s-s_0| < \frac\pi2 e^{b} \mbox{ we have } \Re(s) < \Re(s_0) \right\}.
$$
We see that for all $s \in \mathcal E$ satisfying $\Im(s) < 10^3$ there exist an
$s^\prime \in \mathcal E$ such that\footnote{We have, in
fact, verified this for larger values of $\Im(s)$, but we omit the numerics
here.}
$$
\bigl|s-s^\prime+\pi e^b\bigr| \le 3.
$$
Apparently, related results have been observed in \cite{BW16}.
\begin{table}
\begin{center}
 \begin{tabular}{|c|c|c|}
   \hline 
   \multicolumn{3}{|c|}{Endpoints for strings of zeros with $\Re(z) \approx \delta$} \\
   \hline
   \multicolumn{3}{|c|}{$2b=3\pi$, $z_0=\delta= 0.146949$,
   $\pi\exp(b)= 349.715115$ } \\ 
   \hline
   $k$ & $z_k$ & $z_{k}-z_{k-1}$  \\
   \hline
   1 & $0.146928  + i 351.330281$  &  $ 2.093\cdot 10^{-5} + i 351.33028 $ \\
   2 & $0.146866 + i  702.660561$  &  $ 6.278 \cdot 10^{-5} + i 351.33028 $  \\
   3 & $ 0.146761 + i 1053.990842$ &  $ 1.047 \cdot 10^{-4}+ i 351.33028 $ \\
   \hline
   \multicolumn{3}{|c|}{$2b=8$, $z_0=\delta= 0.172887$,
   $\pi\exp(b)=171.525147$ } \\
   \hline
   $k$ & $z_k$ & $z_{k}-z_{k-1}$ \\
   \hline
   1 & $ 0.172785 + i 172.781$  &   $-1.0196\cdot 10^{-4}+ i 172.781053 $ \\
   2 & $0.172481 +i 346.345$ &  $-3.0451\cdot 10^{-4} + i 173.564643 $ \\
   3 & $0.171974 + i 519.126$ &  $-5.0674\cdot 10^{-4} + i 172.781053 $ \\
   4 & $0.171262 + i 691.907$ &  $-7.1224\cdot 10^{-4} +i 172.781053 $ \\
   5 & $0.170343 + i 865.472$ &  $-9.1839\cdot10^{-4} + i  173.564643 $ \\
   6 & $0.169219 + i 1038.253$ & $-11.2437\cdot10^{-4} + i172.781054 $ \\
   \hline
 \end{tabular}
 \end{center}
\caption{Empirical estimates on zeros near the right boundary $\Re(z)=\delta$.}
\label{table:pearls}
 \end{table}

\begin{figure}[h]
  \begin{center}  
\includegraphics[scale=0.8,angle=0 ]{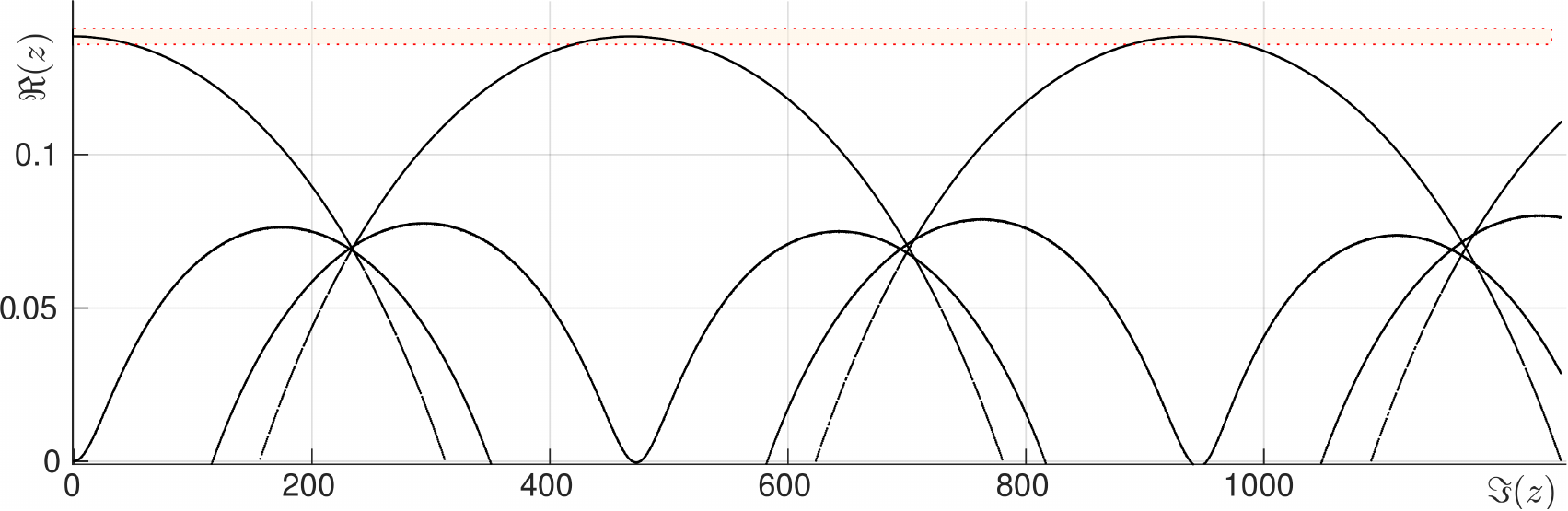} \\
(a) \\
\includegraphics[scale=0.8,angle=0 ]{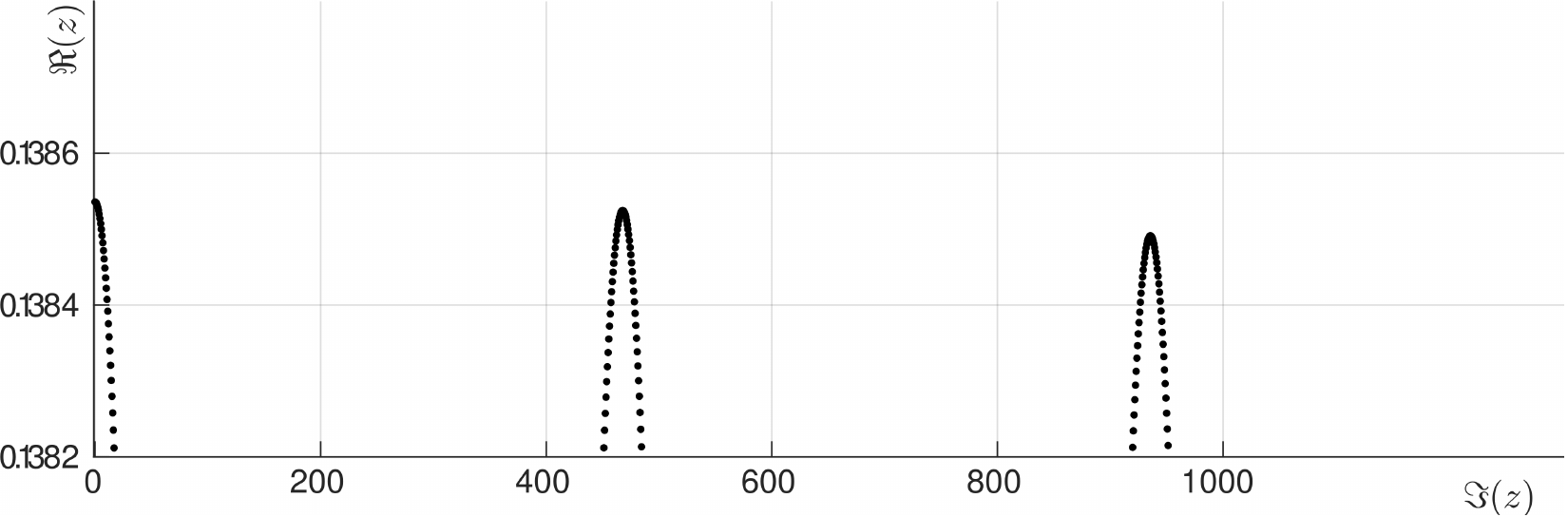} \\
(b) 
\end{center}
\caption[Apparent periodicity]{(a) A plot of the zeros in the case
$\ell(\gamma_{0})=10$; and (b)
 A scaled up version to see the   apparent periodicity near $\Re(s)=\delta$.}
\label{fig:zooming}
\end{figure}

\section{Spacing of zeros}
For completeness, in this Appendix we describe a slight strengthening of a
particular case of a result 
of Weich on the spacing of imaginary parts of zeros. 
The Theorem below asserts that the spacing of the zeros $\mathcal S_{X_b}$ for the 
zeta function is approximately $\frac{\pi}{b}$ as $b \to +\infty$. 
\begin{notation}
A compact part of the critical strip of width $\delta = \delta(b) > 0$ and  height $T$ which we denote by
$$
\mathcal R_b(T) = \{ s \in \mathbb C \mid 0 \leq | \Re(s) | \leq
\delta \mbox{ and }   |\Im(s)| \leq T\}.
$$
We denote a set of regularly spaced points on the lines $\Re(s) =0$ and $\Re(s) =
  \ln 2$ given by $\mathcal L = (\ln 2 + i\pi)\bbZ \cup i\pi\bbZ$.
  \end{notation}

\begin{thm}
  \label{thm:bor2} 
The sets $\mathcal S_{X_b}$ and $\mathcal  L$ are close in
the Hausdorff metric on a small part of the critical strip. 
More precisely, there exists  $\varkappa > 1$ such that  
      $$
      \ddh(b \cdot \left( \mathcal S_{X_b} \cap \mathcal R_b(\varkappa)\right)
      ,\mathcal L)=O\left(\frac{1}{\sqrt b}\right),
      \mbox{ as } b \to +\infty.
      $$ 
\end{thm}
An earlier version of Theorem~\ref{thm:bor2} was established by Weich~\cite{W15},
where he also considered funnels whose widths are $\mathbb Z$-multiples
of $b$. 
On the other hand, his results apply only on a bounded domain and without the error term. 

In order to prove Theorem~\ref{thm:bor2} we need the following approximation
Lemma. 
\begin{lemma}
\label{thm:app1}
Given a $\varkappa>1$ as in Theorem~\ref{thm:main-plus}, the complex analytic function
$Z_{X_b}\bigl(\frac{s}{b}\bigr)$ converges uniformly to $Z_6\bigl(\frac{s}{b}\bigr)$ on the
domain $\mathcal R(\varkappa b)$, more
precisely,
$$
\sup_{  s\in \mathcal R(\varkappa b) }\left|Z_{X_b}\left(\frac{s}{b}\right)
- Z_6\left(\frac{s}{b}\right) \right| =  O\left(\frac{1}{\sqrt b}\right) \mbox{ as }  b \to +\infty.  
$$
\end{lemma}
\begin{proof}
  By a straightforward manipulation using the lengths of the closed geodesics
  estimated in Appendix~\ref{example-ap:s} we show that 
  $$
  a_n(s)=\exp(-nbs-2b)(1+O(\exp(-2b))^s \mbox { for } n=8,10,12
  $$ 
  and therefore 
  $$
  a_n\left(\frac{s}b\right)=\exp(-ns-2b)(1+O(\exp(-2b))^{s/b} \to 0 \mbox{ as }
  b\to\infty.
  $$ 
  The result follows from Theorem~\ref{thm:main-plus} where the corresponding
  terms for $n=8,10$, and $12$ are of order $\left(\frac{1}{\sqrt b}\right)$. 
\end{proof}

Now we are ready to prove Theorem~\ref{thm:bor2}, which is easier than
Theorem~\ref{thm:bor}, because we can use complex analysis.
\begin{proof}
We shall show that on the domain $\mathcal R(\varkappa b)$ we have that the function $Z_{X_b}(s)$ vanishes at
$s_n(b)$ such that $\lim\limits_{b\to\infty} s_n(b)\cdot b = (\ln 2 + i \pi n)$.

  The function $\det(I-e^{-2s}A^2)$ vanishes at $\{i\pi n, \ln 2 +  i\pi n\}$, for $n \in \mathbb Z$.  
For any sufficiently small  $\eta  > 0$ we have that the closed balls 
$$
\overline {U(\ln 2 + 2i \pi n , \eta)} \colon = \{s \in \mathbb C 
\colon |s - (\ln 2 + 2i \pi n) |  \leq \eta \}
$$
contains no more zeros.  Let us denote $\varepsilon = \inf_{s \in \partial U} |\det(I -
e^{-2s}A^2)| >0$. Using Theorem~\ref{thm:app1}, we can now choose $b$
sufficiently large so that  we have 
$$
\inf_{s \in \partial U} \left|Z_{X_b}\left(\frac{s}{b} \right) -  \det(I -
e^{-2s}A^2)\right| < \frac{\epsilon}{2}.
$$
It then follows by Rouch\'e's Theorem~\cite{A78} that for any $n \in \bbN$ the
function $Z_{X_b}(s)$ has exactly one zero $s_n(b)$, satisfying  $\left|s_n(b) -
\frac1b (\ln 2 +  i \pi  n) \right| < \eta$.
\end{proof}

This implies the asymptotic spacing of imaginary parts of zeros.


\begin{thebibliography}{100}
\bibitem{A78} Ahlfors, L. V. Complex analysis. An introduction to the theory of analytic
    functions of one complex variable. Third edition. International Series in
    Pure and Applied Mathematics. McGraw-Hill Book Co., New York, 1978. 
\bibitem{BJ08} Bandtlow, O.~F. and Jenkinson, O. On the Ruelle eigenvalue sequence.
Ergodic Theory Dynam. Systems 28 (2008), no. 6, 1701--1711.     
\bibitem{B83} Beardon,~A.~F. The geometry of discrete groups. Corrected
  reprint of the 1983 original. Graduate Texts in Mathematics, 91. 
  Springer-Verlag, New York, 1995.
  \bibitem{Be} Berger,~M. Lectures on Geodesics Riemannian Geometry, 
   Lectures on mathematics and physics. Mathematics, 33, (Tata Institute of
   Fundamental Research, Mumbai, 1965.
\bibitem{Bohr} Bohr,~H. Almost Periodic Functions. Chelsea Publishing Company,
  New York, N.Y., 1947.  
\bibitem{B14}  Borthwick,~D. Distribution of resonances for hyperbolic
  surfaces. Exp. Math. 23 (2014), no. 1, 25--45. 
  \bibitem{BW16} Borthwick,~D. and Weich,~T.  Symmetry reduction of holomorphic iterated function schemes and factorization
  of Selberg zeta functions, J. Spectral Theory, 6 (2016) 267--329.
\bibitem{Carmo} do Carmo,~M. {\it Riemann Geometry}, Birkhauser, Basel, 1992.
 \bibitem{CE89} Cvitanovi\`c,~P. and Eckhardt,~B. Periodic-orbit quantization of chaotic systems.
Physical review letters 63 (8), 1989, 823--826.
\bibitem{FLP} Fathi,~A., Laundenbach,~F., and Poenaru,~V., 
{\it Travaux de Thurston sur les surfaces}, Ast\'erisque, 66--67, Soc. Math. France, Paris, 1979  
\bibitem{G55} Grothendieck,~A. Produits tensoriels topologiques et espaces
  nucl\`eaires. (French) Mem. Amer. Math. Soc. No. 16 (1955), 140 pp. 
\bibitem{G56}Grothendieck,~A. La th\'eorie de Fredholm. Bull. Soc. Math. France 84 (1956), 319--384. 
\bibitem{GLZ04} Guillop\`e,~L.; Lin, K.~K.; Zworski,~M. The Selberg zeta function for
convex co-compact Schottky groups. Comm. Math. Phys. 245 (2004), no.~1, 149--176. 
\bibitem{hejhal} 
Hejhal,~D. \emph{The Selberg Trace Formula for $PSL(2, \mathbb R)$}, Lecture
Notes in Mathematics 548, Springer, Berlin, 1976. 
\bibitem{JN12} Jakobson,~D. and Naud,~F. On the critical line of convex co-compact hyperbolic surfaces. 
Geom. Funct. Anal. 22 (2012), no. 2, 352--368. 
\bibitem{JN16} Jakobson,~D. and Naud,~F. Resonances and density bounds for convex
co-compact congruence subgroups of $SL_2(\mathbb Z)$. Israel J. Math. 213 (2016), no. 1,
443--473.
\bibitem{JP02} Jenkinson,~O. and Pollicott,~M.
Calculating Hausdorff dimensions of Julia sets and Kleinian limit sets. 
Amer. J. Math. 124 (2002), no. 3, 495--545. 
\bibitem{J54}  Jessen,~B. Some aspects of the theory of almost
  periodic functions. 1957 Proceedings of the International Congress of
  Mathematicians, Amsterdam, 1954, Vol. 1 pp. 305--314 Erven P. Noordhoff N.V.,
  Groningen; North-Holland Publishing Co., Amsterdam.
\bibitem{L80} Levin,~B.~Ja. \emph{Distribution of zeros of entire functions}. 
 Revised edition. Translations of Mathematical Monographs, 5. American
 Mathematical Society, Providence, R.I., 1980. 
\bibitem{Mayer}
 Mayer,~D. The thermodynamic formalism approach to Selberg's zeta function for
 PSL(2,Z). Bull. Amer. Math. Soc. (N.S.) 25 (1991), no. 1, 55--60.   
\bibitem{M98} McMullen,~C.~T. Hausdorff dimension and conformal dynamics. III.
  Computation of dimension. Amer. J. Math. 120 (1998), no. 4, 691--721.
\bibitem{N05}
Naud,~F. Expanding maps on Cantor sets and analytic continuation of zeta
functions, Ann. Sci. Ecole Norm. Sup. 38 (2005), 116--153.
\bibitem{P83} Parry,~W. An analogue of the prime number theorem for closed
  orbits of shifts of finite type and their suspensions. Israel J. Math. 45
  (1983), 41--52. 
\bibitem{PP86} Parry,~W. and Pollicott,~M. The Chebotarev theorem for Galois coverings of Axiom
$A$~flows, Ergodic Theory Dynam. Systems~6 (1986), 133--148.
\bibitem{PP90} Parry,~W. and Pollicott,~M. Zeta functions and the periodic orbit
  structure of hyperbolic dynamics. Ast\'erisque No. 187--188 (1990), 268 pp. 
\bibitem{PP01} Patterson,~S.~J. and Perry,~P. A. The divisor of Selberg's zeta
  function for Kleinian groups. Duke Math. J. 106 (2001), no. 2, 321--390. 
\bibitem{P91}  Pollicott,~M. Some applications of thermodynamic
  formalism to manifolds with constant negative curvature. Adv. Math. 85 (1991),
  no.~2, 161--192. 
\bibitem{R76} Ruelle,~D. Zeta-functions for expanding maps and Anosov
  flows. Invent. Math. 34 (1976), no. 3, 231--242.
\bibitem{S56}  Selberg,~A. Harmonic analysis and discontinuous groups in
  weakly symmetric Riemannian spaces with applications to Dirichlet series. J.
  Indian Math. Soc. (N.S.) 20 (1956), 47--87. 
\bibitem{series} Series,~C., Geometrical Markov coding of geodesics on surfaces
  of constant negative curvature, Ergod. Th. and Dynam. Sys., 
    6 (1986) 601--625.
\bibitem{SL02} Sridhar,~S. and Lu,~W.~T. Sinai billiards, Ruelle zeta-functions
  and Ruelle resonances: microwave experiments. J. Statist. Phys. 108 (2002),
  no. 5--6, 755--765. 
\bibitem{T11} Terras,~A. \emph{Zeta functions of graphs. A stroll through the
  garden.} Cambridge Studies in Advanced Mathematics, 128. Cambridge University
Press, Cambridge, 2011.  
\bibitem{W15}  Weich,~T. Resonance chains and geometric limits on Schottky surfaces. 
Comm. Math. Phys. 337 (2015), no. 2, 727--765.
\end{thebibliography}
\end{document}